%% file: Quasi_Affineness.tex
\pdfoutput=1
\newif\ifpersonal
\documentclass[11pt,a4paper,dvipsnames,x11names]{amsart}

\usepackage{amsmath,amsthm,amssymb,mathrsfs,mathtools,tensor,eucal,thmtools} 
\usepackage[all,cmtip]{xy}

\usepackage[LGR,T1]{fontenc} 
\usepackage[utf8]{inputenc} 
\usepackage{CJKutf8}

\usepackage{microtype,inconsolata} 
\usepackage{enumerate,comment,braket,xspace,tikz,tikz-cd,csquotes} 
\usetikzlibrary{shapes.geometric}
\usepackage[centering,vscale=0.7,hscale=0.7]{geometry}
\usepackage{appendix}
\usepackage[hidelinks,colorlinks=true,linkcolor=Red4,citecolor=RoyalBlue,hypertexnames=false]{hyperref}

\usepackage[capitalize]{cleveref}
\usepackage{xcolor}
\usepackage{xpatch}

\usepackage{mathpazo}
\usepackage{euler}

\usepackage{stackrel}
\usepackage{faktor}
\usetikzlibrary{decorations.markings}

\numberwithin{equation}{subsection}
\theoremstyle{plain}
\newtheorem{theorem}[equation]{Theorem}
\newtheorem{lemma}[equation]{Lemma}
\newtheorem*{maintheorem}{Main Theorem}
\newtheorem*{example*}{Example}
\newtheorem*{definition*}{Definition}
\newtheorem{proposition}[equation]{Proposition}

\newtheorem{corollary}[equation]{Corollary}

\theoremstyle{definition}
\newtheorem{definition}[equation]{Definition}
\newtheorem{notation}[equation]{Notation}
\newtheorem{example}[equation]{Example}
\newtheorem{remark}[equation]{Remark}

\newtheorem{recollection}[equation]{Recollection}
\newtheorem{warning}[equation]{Warning}

\newcommand{\MP}[1]{\textcolor{blue}{#1}}

\input{newcommand}

\begin{document}

\title{(Quasi-)affineness of perverse character varieties}
\author{Enrico Lampetti, Michele Pernice}
\address{Sorbonne Université and Université Paris Cité, CNRS, IMJ-PRG, F-75005 Paris, France}
\email{enrico.lampetti@imj-prg.fr}
\address{University of Washington, Seattle, Washington}
\email{mpernice@uw.com}

\subjclass[2020]{}

\begin{abstract}
    We show that perverse character varieties are (quasi-)affine.
    We do this in a purely stack-theoretic fashion, by exhibiting enough sections of the structure sheaf.
\end{abstract}

\maketitle

\tableofcontents

\section{Introduction}
\addtocontents{toc}{\protect\setcounter{tocdepth}{1}}

	Given a finitely presented group $\Gamma$, the character variety $\mathrm{Char}(\Gamma)$ is an affine scheme parametrizing finite-dimensional semisimple representations of $\Gamma$. The first systematic treatment of character varieties is given by Lubotzky--Magid \cite{LM85}, but they already appeared in the work of Weil \cite{Wei60, Wei62}, Artin \cite{Art69} and Procesi \cite{Pro75}.
	A modern treatment can be found in the work of Sikora \cite{Sik12}. They have been widely studied in the literature and present a rich geometry.
    A central example is the case $\Gamma = \pi_1(X,x)$, where $X$ is a topological space with finitely presented fundamental group. For instance, in the case of $X$ an (open) Riemann surfaces, character varieties give examples of (Poisson, respectively) symplectic varieties \cite{Gol84}.
    
	The so-called monodromy equivalence allows one to interpret $\mathrm{Loc}(X)^\heartsuit \coloneqq \mathrm{Char}(\pi_1(X,x))$ as an algebraic object parametrizing locally constant sheaves on $X$.
    Locally constant sheaves and their moduli arise in many areas of geometry and physics.
	Moreover, locally constant sheaves are related to vector bundles with connections via the Riemann-Hilbert correspondence \cite{Del70} and to Higgs bundles via the non-abelian Hodge correspondence \cite{Sim92}.\medskip
	
	In order to present our main result, we need to take a more intrinsic point of view on character varieties.
	For any $X$ as above, we can construct the stack of locally constant sheaves, or character stack $\mathbf{Loc}(X)^\heartsuit$. The character stack carries more information than the character variety: indeed, while the character variety parametrizes semisimple locally constant sheaves, the stack $\mathbf{Loc}(X)^{\heartsuit}$ parametrizes all locally constant sheaves or, equivalently, all representations of the fundamental group, and its underlying topological space parametrizes conjugacy classes $[\rho]$ of representations of the fundamental group of $X$. Notice that each point has a unique closed point in its closure and such closed point corresponds to the conjugacy class $[\rho^{ss}]$ of its semisimplification, i.e. the associated graded to any Jordan--Holder filtration of $\rho$.
	The class $[\rho^{ss}]$ can be also seen as a point of the character variety $\Loc(X)$ and we get a well defined morphism
\begin{align*}
q \colon \mathbf{Loc} & (X)^\heartsuit \to \Loc(X)^\heartsuit \\
& [\rho] \mapsto [\rho^{ss}] \ .
\end{align*}
	The morphism $q$ defined above is an example of good moduli space, which is the most natural generalization of good quotients in the language of algebraic stacks (see \Cref{sec:gms} for a more detailed discussion). \medskip
	
	The stack $\mathbf{Loc}(X)^\heartsuit$ introduced above is an open substack of a bigger stack $\mathbf{Loc}(X)$, where the latter parametrizes complexes of sheaves on $X$ with locally constant cohomologies. 
    Notice that complexes of sheaves with locally constant cohomologies are a particular example of complexes of sheaves with constructible cohomology and, indeed, the theory generalizes in the constructible setting.
    If $P$ is a stratification of $X$ verifying some mild hypothesis, Haine-Porta-Teyssier in \cite{HPT26} constructed a (derived) stack $\mathbf{Cons}_P(X)$ parametrizing complexes of sheaves with $P$-constructible cohomology.
	Inside this stack, one can find open substacks parametrizing perverse sheaves ${}^{\mathfrak{p}}\mathbf{Perv}_P(X)$, depending on a perversity function $\mathfrak{p} \colon P \to \ZZ$.
	We will refer to these as \textit{perverse character stacks}.\medskip

	An axiomatic background for perverse sheaves has been introduced in the influential paper \cite{BBDG18}.
    The most fundamental and studied examples of perverse sheaves are the one associated to the zero perversity (where one recovers the abelian category of constructible sheaves) and the middle perversity (related to the theory $\mathcal{D}$-modules via a far-reaching generalization of the Riemann--Hilbert correspondence mentioned above \cite{Kas84, Meb84}).
	In both cases, generalizing the aforementioned results for classical character stacks, perverse character stacks admit shifted Poisson structures \cite{CL26} and good moduli spaces ${}^{\mathfrak{p}}\Perv_P(X)$ \cite{Lam25Perv}.
	We will refer to these good moduli spaces as \textit{perverse character varieties}.
	Its points are in bijection with semisimple perverse sheaves.
	In this paper we are concerned with another feature of perverse character varieties, that is, their affineness.
	This leads to the main result of this paper: 
\begin{maintheorem}[{\cref{corollary:quasi_affine_Perv}}]\label{thm_intro:quasi_affineness}
	Perverse character varieties are quasi-affine, that is, the canonical morphism
\[
{}^{\mathfrak{p}}\Perv_P(X) \to \Spec(\Gamma({}^{\mathfrak{p}}\Perv_P(X), \OO_{{}^{\mathfrak{p}}\Perv_P(X)}))
\]
is a quasi-compact open immersion.
\end{maintheorem}

    We want to stress that, although we only prove the quasi-affineness of the good moduli space, the techniques used are completely independent of GIT constructions and thus they fall in the territory of the ``Beyond GIT'' moduli theory: see \cite{AlperHalLei} for a thorough presentation of this more recent intrinsic approach to moduli theory.
    
    The main contribution to the literature is to give an effective way of constructing global sections of the structural sheaf $\OO_{{}^{\mathfrak{p}}Perv_P(X)}$. Informally, to any loop $\gamma$ contained in a stratum $i_p \colon X_p \hookrightarrow X$, we associate the global section that evaluated on a semisimple perverse sheaf $F$ gives back the trace of the monodromy associated to $\gamma$ on the locally constant sheaf $i_p^\ast F$.
	We show that such global sections are enough to separate points, see \Cref{theorem:separating_points_Perv}.
	Both schemes appearing in the Main Theorem above admit natural derived enhancement and we show that the result holds in the derived setting.
	Notice that by Lurie's \cite[Proposition 2.4.2.3]{Lur18}, the derived and underived statements are actually equivalent.
	Since our methods, naturally provide sections of the structural sheaf of the derived enhancement of ${}^{\mathfrak{p}}\Perv_P(X)$, we decided to work directly in the framework of derived algebraic geometry.
	All our results on quasi-affineness holds with the same proof in classical algebraic geometry. \medskip
	
	In what remains, we explain informally the construction of these global sections using Hochschild homology.
    The stack $\mathbf{Cons}_P(X)$ is an example of moduli of objects $\M_\C$ in a stable compactly generated $\infty$-category $\C$ introduced by Toën-Vaquié \cite{TV07}, and revisited by Antieau-Gepner \cite{AG14}.
	More precisely, for $k$ a commutative ring and $\C = \Cons_P(X; \Mod_k)$, there is an equivalence of derived stacks over $k$
\[
\mathbf{Cons}_P(X) \simeq \M_\C \ .
\]
	For any stable compactly generated $\infty$-category $\C$, one can consider its Hochschild homology $\HH(\C)$.
	Hochschild homology can be used to construct $k$-forms on $\M_\C$, $k \in \N$: see for example \cite{BD21}, where a relation between Calabi--Yau structures on $\C$ and shifted symplectic forms on $\M_\C$ is established.
	In the special case of $n=0$, there is a canonically defined morphism
\[
\HO_\C \colon \HH(\C) \to \mathrm{R}\Gamma(\M_\C, \OO_{\M_\C}) \ .
\]
	Hence taking $H^0$ on both sides and setting $\C = \Cons_P(X; \Mod_k)$, we have that Hochschild homology $\HH(\Cons_P(X; \Mod_k))$ gives rise to global sections of the derived stack $\mathbf{Cons}_P(X)$.
    Following \cite{CL26b}, this map can be explicitely described: in \cref{lemma:HH_Cons} we describe $\HH(\Cons_P(X; \Mod_k))$ and in \cref{corollary:HO_Cons} we describe the morphism $\HO_\C$ is this setting.
	Restricting to the open substack ${}^{\mathfrak{p}}\mathbf{Perv}_P(X)$, this gives a map to global sections of the 
    perverse character stack.
	Since by definition the morphism to the good moduli space $q \colon {}^{\mathfrak{p}}\mathbf{Perv}_P(X) \to {}^{\mathfrak{p}}\Perv_P(X)$ induces an equivalence $\OO_{{}^{\mathfrak{p}}\Perv_P(X)} \simeq q_\ast \OO_{{}^{\mathfrak{p}}\mathbf{Perv}_P(X)}$, we have that global sections of the structural sheaf of ${}^{\mathfrak{p}}\mathbf{Perv}_P(X)$ agree with global sections of the structural sheaf of ${}^{\mathfrak{p}}\Perv_P(X)$.
	
\section*{Acknowledgments}

	E.L. thanks his advisors Mauro Porta and Jean-Baptiste Teyssier for their constant support.
	E.L. also wish to thank Peter Haine, Marc Hoyois, Wenyuan Li, Nicolò Sibilla and Sarah Scherotzke for sharing their insights about the content of \cref{subsection:HO}.
	E.L. thanks the University of Washington, where this project was conceived, for the excellent working conditions.
	In particular, E.L. wish to thank Jarod Alper for the hospitality and for his interest and encouragement about the content of the present paper.
    The authors wish to thank Ian Selvaggi for many fruitful discussions.

\addtocontents{toc}{\protect\setcounter{tocdepth}{2}}
\section{Good moduli spaces}\label{sec:gms}

	In this section, we recall the definition of good moduli spaces and their main properties. Good moduli spaces for algebraic stacks were introduced by Alper in \cite{Alp13} and have played a prominent role in moduli theory as generalizations of Mumford's good GIT quotients \cite{Mum94} and Abramovich--Olsson--Vistoli's tame stacks \cite{AOV08}.  
	If a good moduli space $q \colon \X \to X$ exists, then it is unique, as $q$ is universal for maps from $\X$ to algebraic spaces (\cite[Theorem 6.6]{Alp13}).
    Similarly to the Keel--Mori Theorem \cite{KM97}, a fundamental result of Alper--Halpern-Leistner--Heinloth \cite[Theorem A]{AHLH23} provides an intrinsic way to show that an algebraic stack admits a good moduli space.
    The existence of a good moduli space for an algebraic stack $\X$ has many pleasant consequences.
	It allows for example to give a local presentation of $\X$ by quotient stacks and prove the compact generation of $\QCoh(\X)$ (\cite{AHR25}). 
	Also, the existence of good moduli spaces can be exploited to construct BPS Lie algebras and unlock the study of the cohomology of $\X$ via cohomological Hall algebras, see e.g. \cite{DHM22, Dav24, Hen24, BDN+25, HK25}. \medskip

\begin{definition}[{\cite[Definition 4.1]{Alp13}}]
	Let $q\colon \X \to X$ be a qcqs morphism over an algebraic space $S$ with $\X$ an algebraic stack and $X$ an algebraic space. 
	We say that $q \colon \X \to X$ is a good moduli space if the following properties are satisfied:
\begin{enumerate}\itemsep=0.2cm
    \item the functor $q_{\ast} \colon \QCoh(\X) \to \QCoh(X)$ is $t$-exact;
    \item the canonical morphism $\OO_X \to q_\ast \OO_\X$ is an equivalence.
\end{enumerate}
\end{definition}

	For completeness, we recall some of the main properties of good moduli spaces:

\begin{theorem}[{\cite[Proposition 4.5 \& Theorem 4.16 \& Theorem 6.6]{Alp13}}]\label{good_moduli_properties}
	Let $q \colon \X \to X$ be a good moduli space.
	Then
\begin{enumerate}\itemsep=0.2cm
    \item The functor $q^\ast\colon \QCoh(X)^\heartsuit \to \QCoh(\X)^\heartsuit$ is fully faithful;
    \item the map $q$ is universally closed and surjective;
    \item if $Z_1,Z_2$ are closed substacks of $\X$, then
    \[
    \mathrm{Im}(Z_1) \cap \mathrm{Im}(Z_2) = \mathrm{Im}(Z_1 \cap Z_2)
    \]
    where the intersections and images are scheme-theoretic.
    \item for $\kappa$ an algebraically closed field over $k$, the $\kappa$-points of $X$ are $\kappa$-points of $\X$ up to closure equivalence (see \cite[Theorem 4.16-(iv)]{Alp13} for a precise statement);
    \item if $\X$ is reduced (resp., quasi-
compact, connected, irreducible), then $X$ is also.
	If $\X$ is locally noetherian
and normal, then $X$ is also.
    \item if $\X$ is locally noetherian, then $X$ is locally noetherian and $q_{\ast} \colon \QCoh(\X)^\heartsuit \to \QCoh(X)^\heartsuit$ preserves coherence;
    \item if $k$ is excellent and $\X$ is of finite type, then $X$ is of finite type;
    \item if $\X$ is locally noetherian, then $q \colon \X \to X$ is universal for maps to algebraic spaces.
\end{enumerate}
\end{theorem}

\begin{remark}
	The work of Alper--Halpern-Leistner--Heinloth \cite{AHLH23} shows that in order to prove that an algebraic stack $\X$ admits a good moduli space, it is enough to check intrinsic properties of $\X$.
\end{remark}

\begin{definition}
    Let $\pi: \mathcal{X} \rightarrow X$ be a moduli space morphism with $X$ an algebraic space and let $x \in \mathcal{X}$ be a point. We say that $x$ is polystable if $x$ is closed in $\pi^{-1}(\pi(x))$.
\end{definition}

\begin{remark}
    The previous definition is inspired by the definition of polystable vector bundles (and more recently $K$-polystable Fano varieties). Indeed, if $\mathcal{X}$ is locally of finite type over some algebraic space $S$ and admits a good moduli space, a point $x$ of $\mathcal{X}$ is polystable if and only if there are no non-trivial $\Theta$-degenerations of $x$, namely morphisms $\Theta \rightarrow \mathcal{X}$ which maps the generic point of $\Theta$ to $x$. This follows from \cite[Lemma 3.24]{AHLH23}.

    If $\mathcal{X}$ does not admit a good moduli space, the more general definition of $N$-closed point can be given using the notion of $N$-specializations. This will appear in a future work of David Rydh on topological moduli spaces.
\end{remark}
	Finally, let us recall that a notion of derived good moduli spaces has been introduced in \cite{AHPS23}.
	The following is the main result of that paper:
    
\begin{theorem}[{\cite[Theorem 2.12]{AHPS23}}]\label{AHPS23}
	Let $\X$ be a geometric derived stack.
\begin{enumerate}\itemsep=0.2cm
    \item If $t_0\X$ admits a good moduli space $t_0q \colon  t_0\X \to t_0X$, then $\X$ admits a derived good moduli space $q' \colon  \X \to X$ such that $t_0q' \simeq t_0q$.
    \item If $\X$ admits a derived good moduli space $q \colon  \X \to X$, then $t_0q$ is a good moduli space for $t_0X$.
\end{enumerate}
\end{theorem}

\begin{remark}
	When applying \cref{AHPS23} one should take some care in general, as the notion of good moduli space in \cite{AHPS23} is stronger then the one of the one in \cite{Alp13, AHLH23}.
	However, the two notions agree under mild conditions (\cite[Proposition 2.5-(iii)]{AHPS23}) which will always be verified in the cases of interest for us.
\end{remark}

\subsection{Quasi-affineness of good moduli spaces}

    By definition, a good moduli space is only an algebraic space.
    In this paragraph we provide a simple criterion (\cref{corollary:quasi_affineness_good}) for a good moduli space to be a quasi-affine scheme.

\begin{definition}[{\cite[\href{https://stacks.math.columbia.edu/tag/01P6}{Tag 01P6}]{stacks-project} \& \cite[Definition 2.4.1.1]{Lur18}}]
	A (derived) scheme $X$ is called quasi-affine if it is quasi-compact and isomorphic to an open subscheme of a (derived) affine scheme.
\end{definition}

	For a derived scheme $X$, the property of being quasi-affine can be checked on it classical truncation.
	Indeed as a special case of \cite[Proposition 2.4.2.3]{Lur18} we have the following
\begin{theorem}\label{thm:quasi_affine_truncation}
	Let $X$ be a derived algebraic space.
	The following are equivalent:
\begin{enumerate}\itemsep=0.2cm
	\item $X$ is quasi-affine;
	\item $t_0 X$ is quasi-affine.
\end{enumerate}
\end{theorem}
	
\begin{definition}
	Let $Y$ be an algebraic stack and $y_1, y_2 \in Y$.
	We say that a line bundle $\mathcal{L} \in \mathrm{Pic}(Y)$ separates $y_1$ from $y_2$ if there exists $s \in \mathrm{H}^0\mathrm{R}\Gamma(Y, \mathcal{L})$ such that $s_{y_1} \neq 0$ and $s_{y_2} = 0$.
	We say that $\mathcal{L}$ separates points if it separates $y_1$ from $y_2$ for each $y_1, y_2 \in Y$. 
\end{definition}

	The following results are well know to the experts, but we include a proof for completeness

\begin{lemma}\label{lemma:separate_points_is_quasi_finite}
	Let $Y$ be a separated algebraic space locally of finite type over a scheme $S$.
	Assume that the structural sheaf $\OO_Y$ separates points.
	Then
\[
f \colon Y \to \Spec(\mathrm{H}^0\mathrm{R}\Gamma(Y, \OO_Y))
\]
is locally quasi-finite.
\end{lemma}

\begin{proof}
	Given $y_1, y_2 \in Y$ and a global section $s \in \mathrm{H}^0\mathrm{R}\Gamma(Y, \OO_Y)$ be a section separating $y_1$ from $y_2$, we have that the pullback of $X$ and $\Spec(\mathrm{H}^0\mathrm{R}\Gamma(Y, \OO_Y)_s)$ along $\Spec(\mathrm{H}^0\mathrm{R}\Gamma(Y, \OO_Y))$ contains $y_1$ but not $y_2$.
	Hence the fibers $f$ are finite and the result follows by \cite[\href{https://stacks.math.columbia.edu/tag/06RW}{Tag 06RW}]{stacks-project}.
\end{proof}

\begin{lemma}\label{lemma:derived_quasi_affineness}
	Let $Y$ be a derived quasi-compact separated algebraic space locally of finite type over $k$.
	The following are equivalent:
\begin{enumerate}\itemsep=0.2cm
	\item $Y$ is a quasi-affine;
	\item the canonical morphism $f \colon Y \to \Spec(\tau_{\geq 0}\mathrm{R}\Gamma(Y, \OO_Y))$ is representable by a quasi-compact open immersion;
	\item $Y$ is quasi-compact and the structural sheaf $\OO_Y$ separate points of $Y$.
\end{enumerate}
\end{lemma}

\begin{proof}
	The equivalence between $(1)$ and $(2)$ is \cite[Proposition 2.4.1.3]{Lur18} and the fact that $Y$ is connective (see also \cite[Lemma 1.5]{AHPS23}). 
	Clearly $(2)$ implies $(3)$.
	We now show that $(3)$ implies $(2)$.\\ \indent
	Since $Y \to \Spec(k)$ is of locally of finite type, then $f$ is locally of finite type.
	Since being locally quasi-finite is a property defined on the classical truncation of $Y$ (\cite[Definition 3.3.1.1]{Lur18}), if $(3)$ holds then $f$ is locally quasi-finite by \cref{lemma:separate_points_is_quasi_finite}. 
	Moreover since $X$ is quasi-compact and separated and $\Spec(\tau_{\geq 0}\mathrm{R}\Gamma(Y, \OO_Y))$ is affine, we have that $f$ is quasi-compact and separated.
	The result then follows by Zariski's Main Theorem \cite[Theorem 3.3.0.2]{Lur18}.
\end{proof}

	We will use the following criterion for quasi-affineness of a good moduli space.

\begin{corollary}\label{corollary:quasi_affineness_good}
	Let $q \colon \X \to X$ be a good moduli space over $k$ with $X$ separated and of finite type over $k$.
	Then $X$ is quasi-affine if and only if $\OO_\X$ separates polystable points.
\end{corollary}

\begin{proof}
	The points of $X$ are in bijection with the polystable points of $\X$ by \cref{good_moduli_properties}-$(4)$.
	Since $q_\ast \OO_\X \simeq \OO_X$ by definition of $q$, it follows that $\OO_X$ separates points if and only if $\OO_\X$ separates polystable points.
	The claim then follows by \cref{lemma:derived_quasi_affineness}.
\end{proof}
	
\section{Moduli of flat objects}

	In this Section we recall the main definition of moduli of (flat) objects and the main results needed in the rest of the paper.

\subsection{Families of objects}\label{section_moduli_of_objects}

	In this paragraph we recall properties of the tensor product of presentable $\infty$-categories, which is used to define families of objects. 
	We also discuss some finiteness properties one can impose on compactly generated presentable categories.

\begin{recollection}[{\cite[Section 4.8]{Lur17}}]\label{tensor_product}
	For $\C, \D \in \mathrm{Pr}^\mathrm{L}_k$, we can consider their tensor product
\[
\C \otimes_k \D \simeq \FunR_k(\C^{op}, \D)
\]
where the right hand side is the $\infty$-category of $k$-linear functors commuting with limits.
	This tensor product endows $\Pr^{\mathrm{L}}_k$ with a symmetric monoidal structure which restricts to $\Prlo_k$. 
	When $\C$ is compactly generated, there is a canonical equivalence
\[
\C \otimes_k \D \simeq \Funst_k((\C^{\omega})^{op}, \D)
\]
where the right hand side is the $\infty$-category of exact $k$-linear functors.
\end{recollection}

\begin{remark}\label{functoriality_tens_prod}
	Let $\C, \D, \E \in \Pr_k^{\mathrm{L}}$.
	Then an adjunction
\[
f \colon \D \leftrightarrows \E \colon g
\]
induces an adjunction 
\[
\Id_\C \otimes_k f \colon \C \otimes_k \D\leftrightarrows \C \otimes_k \E \colon \Id_\C \otimes_k g.
\]
	The functor $\Id_\C \otimes_k g$ corresponds to 
\[
g \circ - \colon \FunR_k(\C^{op}, \E) \to C \otimes_k \E \simeq \FunR_k(\C^{op}, \D)
\]
under the equivalences of \cref{tensor_product}.
	Its left adjoint $\Id_\C \otimes_k f$ has no easy description in general.
	Notice indeed that post-composition with $f$ do not preserve functors that commutes with limits.
	Nevertheless, if $\C$ is compactly generated one can see that the adjunction $\Id_\C \otimes_k f  \dashv \Id_\C \otimes_k g$ corresponds to 
\[
f \circ - \colon \Funst_k((\C^{\omega})^{op}, \D) \leftrightarrows \Funst_k((\C^{\omega})^{op}, \E) \colon g \circ -
\]
under the equivalences of \cref{tensor_product}.
\end{remark}

\begin{notation}
In the setting of \cref{functoriality_tens_prod}, we denote the functors $\Id_\C \otimes_k f, \Id_\C \otimes_k g$ respectively by $f_\C, g_\C$.
\end{notation}

	In the stable setting, the following result is a reformulation of \cite[Lemma 2.2.1]{SS03}.

\begin{lemma}[{\cite[Lemma 2.4]{AG14}}]\label{generators_stable}
	Let $\E \in \Prlo_k$ stable and $X \subset \E^\omega$ a set of compact objects.
	Then the following are equivalent:
\begin{enumerate}\itemsep=0.2cm
	\item $\forall y \in \E$,
\[
\mathrm{Map}_\E(x,y) \simeq * \in \mathrm{Sp} \ \  \forall x \in X \Rightarrow y \simeq 0;
\]
	\item the $\infty$-category $\E^\omega$ is the smallest full subcategory of $\E$ containing $X$ and stable under finite colimits, shifts and retracts.
\end{enumerate}
\end{lemma}

\begin{definition}\label{compact_gen_def}
	Let $\E \in \Prlo_k$ stable and $X \subset \E^\omega$ a set of compact objects.
	We say that $X$ is a set of compact generators if it satisfies the equivalent conditions of \cref{generators_stable}.
\end{definition}

\begin{example}
	Let $A \in \mathrm{Alg}_k$.
	Then $A$ is a compact generator of $\Mod_A$ in the sense of \cref{compact_gen_def}.
\end{example}

\begin{definition}
	For $\C \in \mathrm{Pr}^\mathrm{L}_k$ and $X\in \dSt_k$, define the $\infty$-category of $X$-families of objects in $\C$ as
\[
\C_X \coloneqq  \C \otimes_k \QCoh(X) \ .
\]
\end{definition}

\begin{notation}
	Let $\C \in \Pr^{\mathrm{L}}_k$ and $S = \Spec (A) \in \dAff_k$.
	The $\infty$-category of $S$-family of objects in $\C$ will be denoted equivalently by $\C_S$ or by $\C_A$.
\end{notation}

\begin{definition}\label{pseudo_perfect}
	For $\C \in \Prlo_k$ and $X \in \dSt_k$, define the $\infty$-category of pseudo-perfect families over $X$ in $\C$ as the full subcategory
\[
\Funst((C^\omega)^{op}, \Perf(X)) \subset \C_X
\]
spanned by functors valued in $\Perf(X)$.
\end{definition}

	In order to introduce finiteness conditions on presentable categories, let us give the following:

\begin{lemma}[{\cite[Theorem D.7.0.7]{Lur18}}]\label{finiteness_conditions}
	Let $\C \in \Prlo_k$.
	Then $\C$ is dualizable in $\Pr_k^{\mathrm{L}}$ with dual given by $\C^\vee \simeq \mathrm{Ind}((\C^\omega)^{op})$.
	In particular, up to a contractible space of choices, there are unique maps in $\Pr^{\mathrm{L}}_k$
\[
\mathrm{ev}_\C \colon  \C^\vee \otimes_k \C \to \Mod_k
\]
\[
\mathrm{coev}_\C  \colon  \Mod_k \to \C \otimes_k \C^\vee
\]
satisfying the triangular identities.
\end{lemma}

	We will need the following finiteness condition:

\begin{definition}\label{def_finiteness_conditions}
	Let $\C \in \Prlo_k$.
	We say that $\C$ is 
\begin{enumerate}\itemsep=0.2cm
    \item smooth if $\ev_\C$ preserves compact objects;
    \item proper if $\coev_\C$ preserve compact objects;
    \item of finite type if it is a compact object of $\Prlo_{k}$.
\end{enumerate}
\end{definition}

	The next result compares the different finiteness conditions of \cref{def_finiteness_conditions}.

\begin{proposition}[{\cite[Proposition II.2.6]{DPS25}}]\label{relation_finiteness_conditions}
	Let $\C \in \Prlo_k$.
	Then:
\begin{enumerate}\itemsep=0.2cm
    \item if $\C$ is of finite type, it is smooth;
    \item if $\C$ is smooth and proper, it is of finite type.
\end{enumerate}
Moreover, if $\C$ is smooth then it admits a compact generator in the sense of \cref{compact_gen_def}.
\end{proposition}

As a consequence of the moreover part, we get the following description of smooth categories.
\begin{lemma}\label{lemma_smoothness}
	Let $\C \in \Prlo_k$.
	The following are equivalent:
\begin{enumerate}\itemsep=0.2cm
	\item $\C$ is smooth;
	\item there exists an equivalence $\C \simeq \LMod_A$ with $A \in \mathrm{Alg}_k$ smooth in the sense of \cite[Definition 4.6.4.13]{Lur17}.
\end{enumerate}
\end{lemma}

\begin{proof}
This follows by combining \cite[Proposition 11.3.2.4]{Lur18} with \cite[Remark 4.2.1.37 \& Remark 4.6.4.15]{Lur17}.
\end{proof}

\begin{remark}\label{rem_smooth}
	Let $\C$ be smooth and $E \in \C$ be a compact generator.
	Lurie shows in the proof of \cite[Proposition 11.3.2.4]{Lur18} that $A \coloneqq \End(E) \in \mathrm{Alg}_k$ is smooth and $\C \simeq \LMod_{A^{rev}}$, where $A^{rev}$ is the opposite algebra of $A$.
\end{remark}

\begin{corollary}[{\cite[Lemma 2.1.13]{Lam25Perv}}]\label{compact_generator_smooth}
	Let $\C \in \Prlo_k$ be a smooth category and $E \in \C$ be a compact generator.
	Let $X$ be a derived stack over $k$.
	Then $F \in \C_X$ is pseudo-perfect if and only if $F(E) \in \Perf(X)$.
\end{corollary}

\begin{lemma}[{\cite[Lemma II.2.20]{DPS25}}]\label{pseudo_perfect_vs_compact}
	Let $\C \in \Prlo_k$ and let $\Spec (A) \in \dAff_k$.
\begin{enumerate}\itemsep=0.2cm
    \item If $\C$ is smooth, then any pseudo-perfect family in $\C_A$ is compact.
    \item If $\C$ is proper, then any compact object in $\C_A$ is pseudo-perfect.
\end{enumerate}
\end{lemma}

\begin{remark}\label{cpt_gen_tens_product}
	Let $\C \in \Prlo_k$ smooth and $\Spec (A) \in \dAff_k$.
	Let $E$ be a compact generator of $\C$ (\cref{relation_finiteness_conditions}).
	Then $E \otimes_k A$ is a compact generator of $\C_A$. Indeed, the functor $\mathrm{Forget}\colon \C_A \to \C$ is the right adjoint of the natural morphism $\C \rightarrow \C_A$ and it is conservative, thus we get the equivalence
	\[
\Map_{\C_A}(E \otimes_k A, -) \simeq \Map_{\C}(E, \mathrm{Forget}(-)) \colon \C_A \to \Mod_k.
\]
\end{remark}

\begin{corollary}[{\cite[Corollary 4.1.24]{Lam25Perv}}]\label{pseudo_perf_Hom}
	Let $\C \in \Prlo_k$ smooth.
	Let $X \in \dSt_k$ and $F_1, F_2 \in \C_X$ pseudo-perfect.
	Then $\Hom_{\C_X}(F_1, F_2) \in \Perf(X)$.
\end{corollary}

\subsection{\texorpdfstring{Moduli of objects and $t$-structures}{Moduli of objects and t-structures}}\label{moduli_of_objects}
	In this paragraph we introduce the moduli of objects of an $\infty$-category of \textit{finite type}.
	This construction firstly appeared in \cite{TV07}, but we will follow the presentation given in \cite[Section 5]{AG14} (see also \cite[Section II.2]{DPS25} and \cite[Section 1.5]{Por25}).
	Notice that our convention for moduli of objects in a category $\C$ differs from the one in \cite{AG14} (but agrees with the one in \cite{TV07, DPS25, Por25}), indeed our moduli of objects of a category $\C$ corresponds the moduli of objects of the linear dual of $\C$ in \cite{AG14}.
	The two points of view are equivalent and do not affect the results nor the proofs in this paragraph.

\begin{recollection}\label{moduli_of_objects_recollection}
	Let $\C \in \Prlo_k$.
	Consider the presheaf
\begin{align*}
	\widehat{\M}_\C \colon &\dAff_k^{op} \to \mathrm{Spc}\\
	& \Spec(A) \mapsto \C_A^{\simeq}.
\end{align*}
	Since $\C$ is dualizable (\cref{finiteness_conditions}), the functor $\C \otimes_k -$ commutes with limits.
	Since $\Mod(-)$ satisfies faithfully flat descent (\cite[Corollary D.6.3.3]{Lur17}), the presheaf $\widehat{\M}_\C $ defines a derived stack over $k$.
	The stack $\widehat{\M}_\C$ is too big and has no chance to be representable.
	The derived moduli of objects of $\C$ is defined as the presheaf 
\begin{align*}
	\M_\C \colon & \dAff_k^{op} \to \mathrm{Spc}\\
	& \Spec(A) \to \C_A^{pp}
\end{align*}
where $\C_A^{pp} \subset \C_A$ is the maximal groupoid spanned by pseudo-perfect families over $A$ in the sense of \cref{pseudo_perfect}.
	Since $\Perf(-)$ satisfies hyper-descent (\cite[Proposition 2.8.4.2-(10)]{Lur18}), the sub-presheaf
\[
\M_\C \subset \widehat{\M}_\C
\]
is a derived substack by \cref{functoriality_tens_prod}.
\end{recollection}

	The main theorem of \cite{TV07} translates as:
\begin{theorem}[{\cite[Theorem 5.8 \& Corollary 5.9]{AG14}}]\label{repr_moduli_of_objects}
	Let $\C \in \Prlo_{k}$ be of finite type.
	Then $\M_\C$ is a locally geometric derived stack locally of finite presentation over $k$.
	Furthermore, the tangent complex at a point $x \colon  \Spec(A) \to \M_\C$ corresponding to a pseudo-perfect family of objects $M_x \in \C_A$ is given by
\[
x^\ast\mathbb{T}_{\M_\C} \simeq \Hom(M_x, M_x)[1].
\]
\end{theorem}

\begin{remark}[{Functoriality of $\M_{(-)}$}]\label{remark:functoriality_moduli_of_objects}
	Let $f \colon \C \to \D$ in $\Prlo_k$ between finite type categories and let $g \colon \D \to \C$ be its right adjoint.
    Assume that $g$ is itself a left adjoint, so that $f$ preserve compact objects.
	Then we have an induced map on moduli of objects
\[
g_{(-)} \colon \M_\D \to \M_\C
\]
	given by $g_A$ on each $\Spec(A) \in \dAff_k$. 
	Indeed this map corresponds to the pre-compoosition with $f$ under the equivalences $\D_A \simeq \Funst_k((\D^{\omega})^{op}, \Mod_A)$ and $\C_A \simeq \Funst_k((\C^{\omega})^{op}, \Mod_A)$. 
\end{remark}

\subsubsection{Moduli of flat objects}\label{subsection:moduli_flat_objects}

	In this paper we are mainly interested in moduli of objects in abelian categories.
	If $\C \in \Prlo_k$ is equipped with a $t$-structure $\tau$, we can introduce a substack 
\[
\M_\C^{[0,0]} \subset \M_\C
\] 
whose classical truncation parametrizes families of objects in the heart $\tau$.
	Under suitable assumption on $\tau$, the stack $\M_\C^{[0,0]}$ is an open substack of $\M_\C$, so that it is also locally geometric and locally of finite presentation when $\C$ is of finite type.

\begin{definition}\label{accessible_t_structure}
	Let $\C \in \Prlo_k$ equipped with a $t$-structure $\tau$. 
	We say that
\begin{enumerate}\itemsep=0.2cm
    \item $\tau$ is $\omega$-accessible if the full subcategory $\C_{\leq 0} \subseteq \C$ is stable under filtered colimits;
    \item $\tau$ is left-complete if $\bigcap \C_{\geq n} \simeq 0$;
    \item $\tau$ is right-complete if $\bigcap \C_{\leq n} \simeq 0$;
    \item $\tau$ is non-degenerate if it is both right-complete and left-complete.
    \item $\tau$ is admissible if it is $\omega$-accessible and non-degenerate.
\end{enumerate}
\end{definition}

\begin{recollection}\label{induced_t_structure}
	Let $\C \in \Prlo_k$.
	Then a morphism $f \colon \Spec (B) \to \Spec (A)$ in $\dAff_k$ corresponding to a map $A \to B$ induces an adjunction
\[
f_\C^\ast \colon \C_{A} \leftrightarrows \C_{B}\colon f_{\C, \ast}
\]
where $f_{\C, \ast}$ is the forgetful functor.
	Denote by $p \colon \Spec(A) \to \Spec(k)$ be the structural morphism.
	As explained in \cite[Section II.2.5.1]{DPS25}, given an $\omega$-accessible right-complete $t$-structure $\tau$ on $\C$, we get an induced $\omega$-accessible right-complete $t$-structure $\tau_A$ on $\C_A$ for any $Spec (A) \in \dAff_k$ by declaring that $F \in \C_A$ is (co)connective if $p_{_\C, \ast} (F) \in \C$ is (co)connective.
	Moreover $\tau_A$ is non-degenerate if $\tau$ is.\\ \indent
	For $X\in \dSt_k$, we can also define an induced $t$-structure $\tau_X$ on $\C_X$ whose connective part is defined by descent.
	That is, $F \in (\C_X)_{\geq 0}$ if and only if $f^\ast F \in (\C_A)_{\geq 0}$ for any $f \colon \Spec (A) \to X$.
\end{recollection}

\begin{definition}[{\cite[Definition II.2.44]{DPS25}}]\label{def_tau_flat}
	Let $\C \in \Prlo_k$ equipped with an $\omega$-admissible right-complete $t$-structure $\tau$ and $\Spec (A) \in \dAff_k$.
	We say that a family $F \in \C_A$ is $\tau$-flat over $\Spec (A)$ (or $\tau_A$-flat) if for every $M \in \Mod_A^\heartsuit$ one has $F \otimes_A M \in \C_A^\heartsuit$.
\end{definition}

\begin{remark}\label{tau_flat_in_heart}
	In the setting of \cref{def_tau_flat}, when $A$ is discrete, a $\tau$-flat family $F$ over $A$ lies in $\C_A^\heartsuit$.
	Indeed in this case one has $A \in \Mod_A^\heartsuit$, so that $F \simeq F \otimes_A A \in \C_A^\heartsuit$.
	The converse holds when $A$ is a field.
\end{remark}

	The key property for a $t$-structure to define an open substack of $\M_\C$ is linked to the following definition:

\begin{definition}
	Let $\C \in \Pr_k^{L, \omega}$ equipped with an admissible $t$-structure $\tau$.
	Let $\S \in \dAff_k$ and let $F \in \C_S$.
	The flat locus of $F$ is the functor 
\[
\Phi_F \colon  (\dAff_k)^{op}_{/_{S}} \to \Spc
\]
\begin{align*}
(T \xrightarrow{f} \S) \mapsto 
\left\{
    \begin{aligned}
        & \ast \ \ \text{if $f^\ast(F)$ is $\tau$-flat relative to $T$,} \\
        & \emptyset \ \ \ \ \text{otherwise.}                  
    \end{aligned}
\right.
\end{align*}
\end{definition}

\begin{definition}[{\cite[Definition II.2.48]{DPS25}}]\label{opennes_flatness}
	Let $\C \in \Pr_k^{L, \omega}$ equipped with an admissible $t$-structure $\tau$.
	We say that $\tau$ universally satisfies openness of flatness if for every $S \in \dAff_k$ and every $F \in \C_S$ the map $\Phi_F \to S$ is representable by an open immersion.
\end{definition}

\begin{proposition}[{\cite[Proposition II.2.56]{DPS25}}]\label{moduli_flat_objects}
	Let $\C \in \Pr_k^{L, \omega}$ be of finite type equipped with an admissible $t$-structure $\tau$.
	Then $\tau$ universally satisfies openness of flatness if and only if then the structural morphism
\[
\M_\C^{[0,0]} \to \M_\C
\]
is representable by a Zariski open immersion.
	In particular, in this case $\M_\C^{[0,0]}$ is a locally geometric derived stack, locally of finite presentation over $k$.
\end{proposition}

\begin{notation}
	In the setting of \cref{moduli_flat_objects}, we set $\M_\C^\heartsuit \coloneqq t_0 \M_\C^{[0,0]}$ to be the classical truncation of the derived stack $ \M_\C^{[0,0]}$.
\end{notation}

	A priori, \cref{moduli_flat_objects} only guarantees that $\M_\C^{[0,0]}$ and its truncation $\M_\C^\heartsuit$ have an open exhaustion by (derived) algebraic stacks.
	When $\Spec (A) \in \dAff_k$ is discrete, the $\infty$-grupoid $\M_\C^{[0,0]}(\Spec(A))$ takes values in $1$-grupoids by \cref{tau_flat_in_heart}.
	Hence $\M_C^{[0,0]}$ is $1$-truncated in the sense of \cite{TV08}.
	Then \cite[Lemma 2.19]{TV07} implies the following

\begin{corollary}\label{geometric_heart}
	Let $\C \in \Pr_k^{L, \omega}$ be of finite type equipped with an admissible $t$-structure $\tau$.
	If $\tau$ universally satisfies openness of flatness, then $\M_\C^{[0,0]}$ is a $1$-Artin stack.
	In particular its truncation $\M_\C^\heartsuit$ is an algebraic stack locally of finite presentation over $k$.
\end{corollary}

\begin{example}
	For an example satisfying universal openness of flatness let us refer to \cref{openness_flatness_perv}.
\end{example}

	Moduli of $\tau$-flat families have affine diagonal, as shown by the following:

\begin{lemma}[{\cite[Lemma 2.4.6]{Lam25Perv}}]\label{affine_diagonal}
	Let $\C \in \Prlo_k$ of finite type equipped with be an admissible $t$-structure $\tau$ universally satisfying openness of flatness.
	Then $\M_\C^\heartsuit$ has affine diagonal.
\end{lemma}

\subsection{Hochschild homology and global sections}\label{subsection:HO}

	We now introduce one the main tools of this paper, that is, Hochschild homology of $k$-linear stable $\infty$-categories.
	This is a generalization of Hochschild homology as classically defined in commutative algebra and scheme theory, and recovers it when one consider the stable category $\Mod_A$ for $A \in \mathrm{CAlg}_k$.
	We thank Marc Hoyois for clarification about the content of this section.\\ \indent
	Recall from \cref{finiteness_conditions} that for any $\C \in \Prlo_k$, there are duality data (unique up to contractible choices)
\[
\ev_\C \colon \C \otimes \C^\vee \leftrightarrows \Mod_k \colon \coev_\C \ ,
\]
where we freely use that Lurie's tensor product is symmetric.
	With this in hand, we can give the following

\begin{definition}
	Let $\C \in \mathrm{Pr}^{\mathrm{L}}_k$ dualizable and $(\ev_\C, \coev_\C)$ a duality datum for $\C$.
	Then the Hochschild homology of $\C$ is defined as the trace of $id_\C$, that is, it is the image of $k$ through the composition
\[
\Mod_k \xrightarrow{\coev_\C} \C \otimes \C^\vee \xrightarrow{\ev_\C} \Mod_k \ .
\]
\end{definition}

	Hochschild homology of a category $\C$ provides a way to construct global section of the structure sheaf of $\M_\C$.
\begin{lemma}\label{lemma:HO}
	Let $\C \in \Prlo_k$ be of finite type.
	There is a canonical map
\[
\HO_\C \colon \HH(\C) \to \mathrm{R}\Gamma (\M_\C, \OO_{\M_\C}) \ .
\]
\end{lemma}

\begin{proof}
	This is well known \cite{BD21, BCS24b, CL26b}, but we include few details for the sake of completeness.
	Since $\C$ is of finite type, we obtain a universal family
\[
F_{univ} \colon (\C^\omega)^{op} \to \Perf(\M_\C)
\]
corresponding to $id_{\M_C}$.
	Since \cref{finiteness_conditions} gives us that $\mathrm{Ind}((\C^\omega)^{op}) \simeq \C^\vee$, we get a natural morphism
\[
\HH(\C) \simeq \HH(\C^\vee) \to \HH(\mathrm{Ind}(\Perf(\M_C))) \ .
\]
	Since the trace on simplicial commutative algebras define a natural transformation $tr \colon \HH(-) \to \OO$ from Hochschild homology to global sections, we obtain the desired map as the composition
\[
\HO_\C \colon \HH(\C) \simeq \HH(\C^\vee) \xrightarrow{F_{univ}} \HH(\mathrm{Ind}(\Perf(\M_C))) \xrightarrow{tr} \mathrm{R}\Gamma(\M_\C, \OO_{\M_\C}) \ .
\]
	Following \cite{BCS24b}, the above map can also be obtained as the composition
\[
\HH(\C) \simeq \HH(\mathrm{Ind}((\C^\omega)^{op})) \to \lim_{\C^\omega \to \Perf(B)} \HH(B) \to \lim_{\C^\omega \to \Perf(B)} B \simeq \mathrm{R}\Gamma(\M_\C, \OO_{\M_\C}) \ ,
\]
where $\C^\omega \to \Perf(B)$ corresponds to a point $\Spec(B) \to \M_\C$ by definition of $\M_\C$.
\end{proof}

\begin{remark}
	Following \cite[Section 4.1]{CL26b}, we have chosen the notation $\HO$ in \cref{lemma:HO} as a mnemonic for "from Hochschild to $\OO$".
\end{remark}

\begin{remark}
	Notice that in the references \cite{BD21, BCS24b}, in the morphism of \cref{lemma:HO} there is no appearance of $\mathrm{Ind}((\C^\omega)^{op}) \simeq \C^\vee$ (\cref{finiteness_conditions}), but only of $\C \simeq \mathrm{Ind}(C^\omega)$.
	This is because they definition of moduli of objects differs from the original one \cite{TV07}, which is the one adopted in the present paper as well as in \cite{CL26b, Lam25Perv}.
	The only difference is that in the aforementioned references, their notion of moduli of objects of $\C$ is the the moduli of objects of $\C^\vee$.
	This doesn't affect the theory as $\HH(\C) \simeq \HH(\C^\vee)$ canonically.
\end{remark}

\begin{remark}
	The formalism of traces \cite{HSS17} allows one to define an action of the circle $S^1$ on Hochschild homology.
	This action plays a crucial in relating Hochschild homology to differential forms on moduli of objects as shown by Brav--Dyckerhoff \cite{BD21}, enhancing the content of \cref{lemma:HO}.
	We do not develop further on this subject, as we shall not need it for our purposes.
\end{remark}

\begin{recollection}\label{recollection:traces}
    Let $B$ be a simplicial commutative $k$-algebra.
    By \cite[Section 4.5]{HSS17} and the equivalence $\Mod_B \simeq \mathrm{Ind}\Perf(B)$, we have an explicit model for $\HH(\Mod_B)$, whose degree $n$-th component is given by
\[
\bigoplus_{M_0, \ldots, M_n \in \Perf(B)} \! \! \! \! \! \! \! \! \! \! \! \Hom_B(M_0, M_1) \otimes_k \cdots \otimes_k \Hom_B(M_{n-1}, M_n) \otimes_k \Hom_B(M_n, M_0) \ ,
\]
and the degeneracy map are the standard ones appearing in the bar complex.
    Since $B$ is a compact generator of $\Mod_B$ and Hochschild homology is Morita invariant, the inclusion of the Bar complex of $B$
\[
i_B \colon \HH(B) \to \HH(\Mod_B) \ ,
\]
sending $B^{\otimes_k (n+1)}$ to the summand of $\HH(\Mod_B)$ given by $B = M_0= \cdots = M_n$, is an equivalence.
    More generally, given a perfect complex $P$, we have that the diagonal inclusion of the full subcategory spanned by $P$
\[
\Mod_{\End_B(P)} \simeq \langle P \rangle \subset \Mod_B \  
\]
induces a "diagonal" inclusion $\HH(\End(P)) \to \HH(\Mod_B)$, which needs not to be an equivalence (unless $P$ is a compact generator).
    Following \cite[Definition 9.5.7 \& Corollary 9.5.8]{Wei94}, for $n \in \N$, consider the morphism $tr^{\otimes n}$
\[
\bigoplus_{M_0, \ldots, M_n \in \Perf(B)} \! \! \! \! \! \! \! \! \! \! \! \Hom_A(M_0, M_1) \otimes_k \cdots \otimes_k \Hom_B(M_{n-1}, M_n) \otimes_k \Hom_B(M_n, M_0) \to B^{\otimes_k (n+1)}
\]
defined on an element $f_0 \otimes_k \cdots \otimes_k f_n$ as the image of $1_B^{\otimes (n+1)} \in B^{\otimes_k (n+1)}$ under the composition
\[
\begin{tikzcd}[sep = scriptsize]
	{B^{\otimes_k (n+1)}} && {(M_0 \otimes_B M_0^\vee) \otimes_k \cdots \otimes_k (M_n \otimes_B M_n^\vee)} && \\
	&& {(M_1 \otimes_B M_0^\vee ) \otimes_k \cdots \otimes_k(M_0 \otimes_B M_n^\vee) } \\
	&& {(M_0 \otimes_B M_0^\vee) \otimes_k \cdots \otimes_k (M_n \otimes_B M_n^\vee)} && {B^{\otimes_k (n+1)} \ .}
	\arrow["{\otimes_k\coev_{M_i}}", from=1-1, to=1-3]
	\arrow["{(f_0 \otimes_B id) \otimes_k \cdots \otimes_k (f_n \otimes_B id)}", from=1-3, to=2-3]
	\arrow["\simeq"{marking, allow upside down}, draw=none, from=2-3, to=3-3]
	\arrow["{\otimes_i \ev_{M_i}}", from=3-3, to=3-5]
\end{tikzcd}
\]
Using that the composition morphism $\Hom_B(M_0, M_1) \otimes_k \Hom_B(M_1, M_2) \to \Hom_B(M_0, M_2)$ can be identified with the morphism 
\[
\begin{tikzcd}[column sep=tiny,row sep=small]
	{(M_0^\vee \otimes_B M_1) \otimes_k (M_1^\vee \otimes_B M_2)} && {(M_1^\vee \otimes_B M_1) \otimes_k (M_0^\vee \otimes_B M_2) } && \\
	\\
	&& { B \otimes_k (M_0^\vee \otimes_A M_2)} && {M_0^\vee \otimes_B M_2,}
	\arrow["\sim", from=1-1, to=1-3]
	\arrow["{\ev_{M_1} \otimes_k id}", from=1-3, to=3-3]
	\arrow["m", from=3-3, to=3-5]
\end{tikzcd}
\]
one checks that the morphisms $tr^{\otimes n}$ assemble to a well defined map
\[
tr \colon \HH(\Mod_B) \to \HH(B) \ .
\]
    Since clearly $tr \circ \HH(i) \simeq id_{\HH(B)}$ and since $\HH(i)$ is an equivalence, it follows that $tr$ is an equivalence too, inverse of $\HH(i)$.
\end{recollection}

    We extract the following definition from \cref{recollection:traces}:

\begin{definition}\label{definition:trace}
    Let $B$ be simplicial commutative algebra and let $M \in \Perf(B)$.
    Given $f \colon M \to M$, we define the trace of $f$, denoted $tr(f)$, as the image of the unit of $B$ under the map
\[
B \to M \otimes_B M^\vee \xrightarrow{f \otimes id} M \otimes_B M^\vee \to B \ ,
\]
where the first and third map are respectively the coevaluation and evaluation of a duality datum for $M$.
\end{definition}

\begin{notation}\label{notation:pseudo_perfect}
	Let $\C \in \Prlo_k$ of finite type (hence smooth by \cref{finiteness_conditions}) and write $\C \simeq \mathrm{LMod}_A$, $A \in \mathrm{Alg}_k$ smooth (\cref{lemma_smoothness}). 
    Let $B \in \mathrm{CAlg}_k$ be a commutative $k$-algebra and $F \in (\LMod_A)_B \simeq \LMod_{A \otimes_k B}$ pseudo-perfect.
\[
p_\ast \colon \C \otimes_k B \simeq \LMod_{B \otimes_k A} \simeq \Mod_B(\LMod_A) \to \Mod_B \ .
\]
	By definition of pseudo-perfect object, we have that $p_\ast F \in \Perf(B)$.
    Moreover, the object $F \in \LMod_{A \otimes_k B}$ is classified by a $k$-linear map
\[
A \to \End_B(p_\ast M) \ .
\]
\end{notation}

	The following result is a mild generalization of Casals-Li's \cite[Proposition 4.5]{CL26b}.

\begin{proposition}\label{proposition:HO_description}
	In the setting of \cref{notation:pseudo_perfect}, let $x \colon \Spec(B) \to \M_\C$ be a point corresponding to a pseudo-perfect left $A \otimes_k B$-module $(p_\ast F, \rho \colon A \to \End_B(p_\ast F))$ as in \cref{notation:pseudo_perfect}.
	Then for $r \in A$, we have that its image on $B$ under the composite 
\[
\HO_\C \colon \HH(\C) \simeq \HH(A) \to \mathrm{R}\Gamma(\M_\C,\OO_{M_C}) \xrightarrow{x^\ast} B
\]
is given by $x^\ast \HO(r)(x) = tr(\rho(r))$.
\end{proposition}

\begin{remark}
    In \cref{proposition:HO_description}, we consider $r \in A$ as an element in $\HH(A)$.
    We will be interested into the non-derived global sections of $\OO_{\M_\C}$.
    If $A$ is connective, which will be the case in the applications of this paper, \cref{proposition:HO_description} implies the commutativity of the following diagram:
\[\begin{tikzcd}[sep=small]
	A && {\End_{\OO_{\M_{\C}}}({p_* F}_{univ})} \\
	\\
	{\HH_0(A) \simeq \pi_0(A)/[\pi_0(A),\pi_0(A)]} && {\Gamma(\M_{\C}, \OO_{\M_{\C}})}
	\arrow["{\rho_{univ}}", from=1-1, to=1-3]
	\arrow[two heads, from=1-1, to=3-1]
	\arrow["tr", from=1-3, to=3-3]
	\arrow[from=3-1, to=3-3]
\end{tikzcd}\]
where $(p_\ast F_{univ}, \rho_{univ})$ is the universal $A \otimes_k \OO_{\M_\C}$ pseudo-perfect complex, i.e., the universal family over $\M_\C$.
    The aforementioned commutativity is enough for all our purposes.
\end{remark}

\begin{proof}
    Consider the forgetful functor 
\[
p_\ast \colon \C \otimes_k B \simeq \LMod_{B \otimes_k A} \simeq \Mod_B(\LMod_A) \to \Mod_B \ .
\]
    By definition of $\M_\C$, we have that a point $x \colon \Spec(B) \to \M_\C$ corresponds to an object $F \in \C \otimes_k B$ such that $p_\ast F \in \Perf(B)$.
    We have that $F$ corresponds to the datum of a $k$-linear map
\[
\rho \colon A \to \End_B(p_\ast F) \ .
\]
    Under the equivalence $\LMod_{B \otimes_k A} \simeq \Fun^{ex}_k((\LMod_A^\omega)^{op}, \Mod_B)$, the object $F$ corresponds to the functor
\begin{align*}
    h_F \colon (\LMod&_A^\omega)^{op} \to \Mod_B\\
    & M \mapsto \Hom_{A \otimes_k B}(M\otimes_k B, F) \ .
\end{align*}
    In particular, we get that $h_F(A) \simeq \Hom_{B}(B, p_\ast F) \simeq p_\ast F$.
    Hence we have a commutative diagram
\[
\begin{tikzcd}[sep=small]
	{(\LMod_A^\omega)^{op}} && {\Mod_B} \\
	\\
	{\langle A \rangle} && {\langle p_\ast F \rangle}
	\arrow["{h_F}", from=1-1, to=1-3]
	\arrow[hook, from=3-1, to=1-1]
	\arrow["{\langle \rho \rangle}"', from=3-1, to=3-3]
	\arrow[hook, from=3-3, to=1-3]
\end{tikzcd}
\]
where the vertical arrows are the inclusion of the full subcategory spanned respectively by $A \in \Mod_A$ and $p_\ast F \in \Mod_B$, and the morphism $\langle \rho \rangle$ is the one determined by $\rho$ (it sends $A$ to $p_\ast F$ and acts as $\rho$ on $A \simeq \End_A(A) \to \End_B(p_\ast F)$.
    Since $p_\ast F \in \Perf(B)$, we get that $\End_B(p_\ast F) \in \Perf(B)$.
    Hence by \cref{recollection:traces} we have a commutative diagram
\[
\begin{tikzcd}[sep=small]
	{\HH(\Mod_A) \simeq \HH(\LMod_A^\omega)^{op})} && {\HH(\Mod_B)} && {HH(B)} \\
	\\
	{\HH(A)} && {\HH({\End_B(p_{\ast} F)})}
	\arrow["{\HH(h_F)}", from=1-1, to=1-3]
	\arrow["tr", from=1-3, to=1-5]
	\arrow[from=3-1, to=1-1]
	\arrow["{\HH(\rho)}"', from=3-1, to=3-3]
	\arrow[from=3-3, to=1-3]
\end{tikzcd}
\]
    where the left vertical morphism is an equivalence.
    The statement follows by the explicit description of the equivalence given by the trace $tr$ appearing in \cref{recollection:traces}.
\end{proof}

	We now want to understand the behavior of $\HO \colon \HH_0(\C) \to H^0\mathrm{R}\Gamma(\M_\C,\OO_{M_C})$ under récollements.
	This will be useful for us, as categories of constructible sheaves arise in this way.
	We mainly follow \cite{BH26}, where the results treated here are proven in greater generality.
	
\begin{definition}[{\cite[Definition A.8.1]{Lur17}}]\label{def:recollements}
	Let $\C$ be an $\infty$-category with finite limits.
	We say that functors
\[
i^\ast \colon \C \to \Z \qquad \text{and} \qquad j^\ast \colon \C \to \U
\]
exhibit $\C$ as the récollement of $\Z$ and $\U$ if the following conditions hold:
\begin{itemize}\itemsep=0.2cm
	\item The functors $i^\ast$ and $j^\ast$ admit fully faithful right adjoints $i_\ast$ and $j_\ast$, respectively.
	\item The functors $i^\ast \colon \C \to \Z$ and $j^\ast \colon \C \to \U$ are left exact and jointly conservative.
	\item The functor $j^\ast i_\ast \colon \Z \to \U$ is constant with value the terminal object of $\U$.
\end{itemize}
	We also simply say that $(i^\ast \colon \C \to \Z, j^\ast \colon \C \to \U)$ is a recollement to mean that $\C$ admits finite limits and $i^\ast$ and $j^\ast$ exhibit $\C$ as a recollement of $\U$ and $\Z$.
\end{definition}

\begin{remark}\label{remark:special_adjoints}
	In the setting of \cref{def:recollements}, if $\C$ is stable (respectively, presentable), then $\U$ and $\Z$ are also stable (respectively, presentable) by\cite[Notation 1.11 \& Corollary 1.16]{BH26}.
	Moreover by \cite[Remark A.8.19]{Lur17}, in this case $i_\ast$ has a right adjoint $i^!$ and $j^\ast$ has a left adjoint $j_\#$.
	Explicitly, we have
\[
i^! \simeq \fib(i^\ast \to i^\ast j_\ast j^\ast), \qquad j_\# \simeq \cof(j^\ast \to i_\ast i^\ast j^\ast) \ .
\]
\end{remark}

\begin{definition}
    Let $(i^\ast \colon \C \to \Z, j^\ast \colon \C \to \U)$ be a récollement.
    We say that it is a récollement in $\Pr_k^{\mathrm{L}}$ if $\C, \Z, \U \in \Pr_k^{\mathrm{L}}$ and each functor appearing in \cref{def:recollements} is in $\Pr_k^{\mathrm{L}}$.
\end{definition}

	\cref{remark:special_adjoints} shows that there are many adjoints functors when dealing with récollements of stable $\infty$-category.
	This allows us to prove the following
	
\begin{lemma}\label{lemma:finite_type_récollement}
	Let $(i^\ast \colon \C \to \Z, j^\ast \colon \C \to \U)$ be a récollement in $\Pr_k^{\mathrm{L}}$. 
\begin{enumerate}\itemsep=0.2cm
	\item If $\C$ is dualizable, then $\Z, \U \in \Pr_k^{\mathrm{L}}$ are dualizable.
	\item if $\C, \Z, \U \in \Prlo_k$ and $\C$ is of finite type, then $\U, \Z \in \Prlo_k$ are of finite type.
\end{enumerate}
\end{lemma}

\begin{proof}
	As shown in \cite[Corollary 1.16]{BH26} and its proof, the functors $i_\ast, i^\ast$, $j_\ast$ and $j^\ast$ exhibit $\Z$ and $\U$ as retracts of $\C$ in $\Pr_k^{\mathrm{L}}$.
	Item $(1)$ follows.\\
	If $j_\ast$ admits a right adjoint, then it commutes with colimits and moreover $j^\ast$ preserve compact objects.
	Since $i^! \simeq \fib(i^\ast \to i^\ast j_\ast j^\ast)$ by \cref{remark:special_adjoints}, we have that $i^!$ commutes with filtered colimits.
	It follows $i_\ast$ preserve compact objects.
	Hence the functors $i_\ast, i^\ast$, $j_\#$ and $j^\ast$ exhibit $\Z$ and $\U$ as retracts of $\C$ in $\Prlo_k$.
	Item $(2)$ follows.
\end{proof}

\begin{lemma}\label{lemma:splitting_HH}
	Let $(i^\ast \colon \C \to \Z, j^\ast \colon \C \to \U)$ be a récollement $\Pr_k^{\mathrm{L}}$.
	Then the morphisms in $\Mod_k$
\[
\HH(i^\ast) \oplus \HH(j^\ast) \colon \HH(\C) \to \HH(\Z) \oplus \HH(\U)
\]
\[
\HH(i_\#) \oplus \HH(j_\#) \colon \HH(\Z) \oplus \HH(\U) \to \HH(\C)
\]
are inverse equivalences.
\end{lemma}

\begin{proof}
	The proof is the same as \cite[Proposition 3.2]{BH26}.
	The first map is an equivalence by \cite[Lemma 1.17]{BH26}.
	By fully faithfulness of $i_\ast$ and $j_\ast$ (hence of $i_\#$ and $j_\#$), we have $i^\ast i_\ast \simeq i^\ast i_\# \simeq id$ and $j^\ast j_\ast \simeq j^\ast j_\# \simeq id$.
	Hence the composite $(\HH(i^\ast) \oplus \HH(j^\ast)) \circ \HH(i_\#) \oplus \HH(j_\#) \simeq id_{\HH(\C)}$.
	Since the $\HH(i^\ast) \oplus \HH(j^\ast)$ is an equivalence, we deduce that $\HH(i_\#) \oplus \HH(j_\#)$ is its inverse equivalence.
\end{proof}

In the next result, we keep the notations of \cref{notation:pseudo_perfect} and \cref{proposition:HO_description}.

\begin{corollary}\label{lemma:splitting_HO_description}
	In the setting of \cref{lemma:finite_type_récollement}-$(2)$, let $R, R_U, R_Z \in \mathrm{CAlg}_k$ smooth algebras such that we have equivalences $\C \simeq \LMod_R$, $\U \simeq \LMod_{R_\U}$, $\Z \simeq \LMod_{R_\Z}$.
	Under the equivalence of \cref{lemma:splitting_HH}, the morphism $\HO \colon \HH_0(\C) \to \mathrm{H}^0\mathrm{R}\Gamma(\M_\C, \OO_{\M_\C})$ is given by
\begin{align*}
\HO \colon & \HH_0(\U) \oplus \HH_0(\Z) \to \mathrm{H}^0\mathrm{R}\Gamma(\M_\C, \OO_{\M_\C}) \\
& HO([r_\U], [r_\Z])([x]) = tr(\rho_U(r_u)) + tr(\rho(r_\Z)) \ ,
\end{align*}
where $x=(V \in \Perf_k, \rho \colon R \to \End_k(V))$ corresponds to a point $x \colon \Spec(k) \to \M_\C$, and similarly $(V_Z, \rho_Z)$, $(V_U, \rho_U)$ are the objects corresponding to the point $(i^\ast x, j^\ast x) \colon \Spec(k) \xrightarrow{x}\M_\C \xrightarrow{(i^\ast, j^\ast)} \M_\Z \times \M_\U$.
\end{corollary}

\begin{proof}
	By \cref{lemma:finite_type_récollement} and \cref{repr_moduli_of_objects}, the derived stacks $\M_\C, \M_\Z, \M_\U$ are locally geometric locally of finite presentation over $k$.
	Consider the morphism of derived stacks over $k$
\[
(i^\ast, j^\ast) \colon \M_\C \to \M_\Z \times \M_\U
\]
	of \cref{remark:functoriality_moduli_of_objects}.
	One immediately verifies that the morphism $\HO$ of \cref{lemma:HO} is functorial, so that we get a commutative diagram
\begin{equation}\label{diagram:HO_commutative}
\begin{tikzcd}[sep=small]
	{\HH(\Z) \oplus \HH(\U)} &&& {\mathrm{H}^0\mathrm{R}\Gamma(\M_\Z, \OO_{\M_\Z}) \oplus \mathrm{H}^0\mathrm{R}\Gamma(\M_\U, \OO_{\M_\U})} \\
	\\
	{\HH(\C)} &&& {\mathrm{H}^0\mathrm{R}\Gamma(\M_\C, \OO_{\M_\C})}
	\arrow["{\HO_\Z \oplus \HO_\U}", from=1-1, to=1-4]
	\arrow["{\HH(i_\#) \oplus \HH(j_\#)}"', from=1-1, to=3-1]
	\arrow["{(i^\ast)^\# \oplus (j^\ast)^\# }", from=1-4, to=3-4]
	\arrow["{\HO_\C}"', from=3-1, to=3-4]
\end{tikzcd}
\end{equation}
where the left vertical map is an equivalence by \cref{lemma:splitting_HH}.
\end{proof}

\section{Moduli of constructible and perverse sheaves}

	In this paragraph we introduce constructible and perverse sheaves, as well as their moduli. The role of perverse sheaves is central in the study of singularities and they naturally generalize local systems.
	Moreover they are linked to the study of differential equation with regular singularities (i.e., simple poles) via the Riemann--Hilbert correspondence (\cite{Kas84, Meb84}).
	Perverse sheaves originated form the theory of $\mathcal{D}$-modules via the Riemann--Hilbert correspondence as the shifted solution complex and the shifted de Rham complex of a holonomic $\mathcal{D}$-module on a complex manifold (\cite{Kas84, KS90}).
	An axiomatic background has been introduced in the influential paper \cite{BBDG18}
    We claim no originality here, possibly except \cref{subsection:HO_Cons}.
	Every statement claimed in this paragraph holds true in much wider generality, for which we refer to  \cite{PT22, HPT24, HPT26}.
	We choose to work under more restrictive hypothesis so that we can avoid introducing many technical notions, and in any case this is enough for our scope.

\subsection{Constructible and Perverse sheaves}

\begin{warning}
	Throughout this section we will work with hypersheaves instead of sheaves.
	In the application we are interested in, there is no difference between the two notions.
	Since there is no risk of confusion, we will not make any reference to this in our notation, see \cite{MO:168526}.
\end{warning}

\begin{definition}
	A stratified space is the data of a continuous map $\rho \colon X \to P$ where $X$ is a topological space and $P$ is a poset endowed with the Alexandroff topology (\cite[Notation 5.1]{HPT23}).
\end{definition}

\begin{definition}
	A complex subanalytic stratified space is the data of a triple $(M,X,P)$ where $M$ is a smooth complex analytic space, $X \subset M$ is a locally closed complex subanalytic subset, and $X \to P$ is a locally finite stratification by complex subanalytic subsets of $M$.
\end{definition}

\begin{definition}
	A complex algebraic stratified space is the data of a stratified space $(X,P)$ where $X$ is (the complex points of) an algebraic variety over $\mathbb{C}$ and $X \to P$ is a finite stratification by Zariski locally closed subsets.
\end{definition}

\begin{definition}\label{Whitney_strat_def}
	Let $M$ be a smooth manifold. 
	A stratified space $(X, P)$ on a closed subset $X \subset M$ is said to be a Whitney stratified space if the following conditions hold:
	\begin{enumerate}\itemsep=0.2cm
    \item $P$ is finite;
    \item for every $p \in P$, the stratum $X_p$ is smooth and connected;
    \item for every $p,q \in P$, $p < q$, the pair $(X_p, X_q)$ satisfies Whitney's conditions $B$ (see e.g. \cite[Section 1.2]{GM88}).
    \item if $p \in P$ and $\overline{X_p}$ is the closure of $X_p$ in $M$, the intersection $(\overline{X_p} \setminus X_p) \cap X$ is a union of strata.
\end{enumerate}
\end{definition}

\begin{definition}
	Let $(X,P)$ be a stratified space and let $\E \in \Cat_\infty$. 
	For $p \in P$, denote by $i_p \colon X_p \to X$ the inclusion.
	An object $F \in \Sh(X; \E)$ is constructible if $i_p^\ast F$ is locally constant for every $p \in P$.
	We denote by
	\[
	\Cons_P(X; \E) \subset \Sh(X ; \E)
	\]
	the full subcategory spanned by constructible sheaves.
\end{definition}

\begin{definition}
	Let $(X,P)$ be a stratified space and let $\E \in \Cat_\infty$.
	We let 
\[
\Cons_{P, \omega}(X; \E) \subset \Cons_P(X; \E)
\]
be the full subcategory spanned by constructible sheaves with compact stalks.
\end{definition}

	The following shows that the stack of constructible sheaves, constructed as a moduli of objects is locally geometric.

\begin{proposition}\label{proposition:Whitney_finite_type}
	Let $(X,P)$ be a complex Whitney stratified space.
	Then $\Cons_P(X; \Mod_k) \in \Prlo_k$ is of finite type.
\end{proposition}

\begin{proof}
	This follows by combining \cite[Corollary 4.1.11]{HPT26} with \cite[Theorem 4.1.22]{Lam25Perv}.
\end{proof}

	The following statement will allow us to apply the results from \cref{subsection:HO}.

\begin{proposition}\label{proposition:recollements_Cons}
	Let $(X,P)$ be a complex Whitney stratified space and let $Z \subset P$ be a closed subset and let $U \coloneqq P \setminus Z$.
	Consider the corresponding open and closed immersions $j \colon X_U \hookrightarrow X$ and $i \colon X_Z \to X$.
	Let $A$ be a simplicial commutative ring.
	The functors
\[
i^\ast \colon \Cons_P(X; \Mod_A) \to \Cons_Z(X_Z; \Mod_A), \ \ \ \ j^\ast \colon \Cons_P(X; \Mod_A) \to \Cons_U(X_U; \Mod_A) 
\]
exhibits $\Cons_P(X; \Mod_A)$ as a récollement of $\Cons_Z(X_Z; \Mod_A)$ and $\Cons_U(X_U; \Mod_A)$.
	Moreover, the above récollement restricts to exhibit $\Cons_{P, \omega}(X; \Mod_A)$ as a récollement of $\Cons_{Z, \omega}(X_Z; \Mod_A)$ and $\Cons_{U, \omega}(X_U; \Mod_A)$.
\end{proposition}

\begin{proof}
	By \cite[Theorem 4.1.22]{Lam25Perv}, this is a particular case of \cite[Recollection 4.2.3]{Lam25Perv}.
	See also \cite[Corollary 6.7.2 \& Corollary 6.7.6]{PT22}).
\end{proof}

\begin{notation}
	In the setting of \cref{proposition:recollements_Cons}, the left adjoint $j_\#$ of $j^\ast$ arising from the recollement is the extension by zero functor and usually denoted by $j_!$.
	We will adopt this convention.
\end{notation}

	We now introduce the categories of perverse sheaves, following \cite{BBDG18}.

\begin{definition}\label{def_perv}
	Let $(X,P)$ be a stratified space and $\mathfrak{p} \colon P \to \ZZ$ be a function. 
	Let $A$ be a simplicial commutative ring.
	Consider the pair of full subcategories $\Sh(X;\Mod_R)$ defined by
\[
{}^{\mathfrak p} \Sh(X;\Mod_A)_{\geq 0} \coloneqq \big\{F \in \Sh(X;\Mod_A) \mid \ \forall p \in P \ , \ \pi_i( i_p^{\ast}(F) ) = 0 \text{ for every } i \leq \mathfrak p(p) \big\},
\]
\[
{}^{\mathfrak p} \Sh(X;\Mod_A)_{\leq 0} \coloneqq \big\{F \in \Sh(X;\Mod_A) \mid \ \forall p \in P \ , \ \pi_i( i_p^{!}(F) ) = 0 \text{ for every } i \geq \mathfrak p(p) \big\}.
\]
	The $\infty$-category of perverse sheaves is
\[
{}^\mathfrak{p} \Perv(X; \Mod_A) \coloneqq {}^{\mathfrak p} \Sh(X;\Mod_A)_{\leq 0} \cap {}^{\mathfrak p} \Sh(X;\Mod_A)_{\geq 0} \ .
\]
\end{definition}

\begin{notation}
	In the setting of \cref{def_perv}, we will refer to the function $\mathfrak{p} \colon P \to \ZZ$ as the \textit{perversity}.
\end{notation}

\begin{definition}
    In the setting of \cref{def_perv}, we set
\[
{}^{\mathfrak p} \Cons_P(X;\Mod_R)_{\geq 0} \coloneqq {}^{\mathfrak p} \Sh(X;\Mod_R)_{\geq 0} \cap \Cons_P(X;\Mod_R) \subset \Cons_P(X;\Mod_R) \ .
\]
    We define analogously ${}^{\mathfrak p} \Cons_P(X;\Mod_R)_{\leq 0}$, ${}^{\mathfrak p} \Perv_P(X;\Mod_R)$, ${}^{\mathfrak p} \Cons_{P, \omega}(X;\Mod_R)_{\geq 0}$, ${}^{\mathfrak p} \Cons_{P, \omega}(X;\Mod_R)_{\leq 0}$, ${}^{\mathfrak p} \Perv_{P, \omega}(X;\Mod_R)$.
\end{definition}

\begin{proposition}\label{perverse_t_structure}
	Let $(X,P)$ be a stratified space with $P$ finite and $A$ be a simplicial commutative ring.
\begin{enumerate}\itemsep=0.2cm
    \item The pair of $\infty$-categories $({}^{\mathfrak p} \Sh(X;\Mod_A)_{\leq 0}, {}^{\mathfrak p} \Sh(X;\Mod_A)_{\geq 0})$ of \cref{def_perv} define a $t$-structure ${}^{\mathfrak{p}} \tau$ on $\Sh(X;\Mod_A)$.
\end{enumerate}  
    If furthermore $(X,P)$ is a complex Whitney stratified space, then:
\begin{enumerate}\itemsep=0.2cm
    \item[(2)] the $t$-structure ${}^{\mathfrak{p}} \tau$ restricts to a $t$-structure on $\Cons_P(X;\Mod_R)$;
    \item [(3)]if moreover $A$ is a discrete regular ring, the $t$-structure ${}^{\mathfrak{p}} \tau$ restricts to a $t$-structure on $\Cons_{P, \omega}(X;\Mod_A)$.
\end{enumerate}
\end{proposition}

\begin{proof}
	By \cite[Theorem 4.1.22]{Lam25Perv}, this is a special case of \cite[Proposition 4.2.5]{Lam25Perv}.
\end{proof}

	We now introduce the two perversities that are relevant for this paper.

\begin{definition}\label{definition:middle_perversity}
	Let $(X,P)$ be a complex Whitney stratified space.
	The middle perversity is the function
\begin{align*}
	\mathfrak{p} \colon & P \to \ZZ \\
	& p \mapsto \dim_{\mathbb{C}} X_p \ .
\end{align*}
\end{definition}	

\begin{definition}\label{definition:zero_perversity}
	Let $(X,P)$ be a stratified space with $P$ finite.
	The trivial perversity is the constant function $P \to \ZZ$ equal to zero.
\end{definition}

\begin{remark}\label{remark:trivial_perversity}
	In the setting of \cref{definition:zero_perversity}, the corresponding perverse $t$-structure given by \cref{perverse_t_structure}-$(1)$ is the standard $t$-structure (\cite[Lemma 3.1.8]{HPT26}), for which $F \in \Sh(X; \Mod_A)$ is (co)connective if and only if its stalks are (co)connective $A$-modules.
	This $t$-structure always descends to $\Cons_P(X; \Mod_A)$ (and to $\Cons_{P, \omega}(X; \Mod_A)$ if $A$ is a discrete regular ring).
	An element of $F \in \Cons_P(X; \Mod_A)$ is ${}^{\mathfrak{p}}\tau$-flat if and only if its stalks are $A$-modules of Tor-amplitude $[0,0]$.
\end{remark}

\begin{notation}
	In the setting of \cref{definition:zero_perversity}, we denote the perverse $t$-structure by $\tau_{st}$ and its heart by $\Cons_P(X; \Mod_A)^\heartsuit$.
\end{notation}

	\cref{proposition:recollements_Cons} and \cref{perverse_t_structure} allow one to give an explicit description of simple perverse sheaves with perfect stalks, that we now recall.
	This will play a crucial role in \cref{subsection:quasi_affineness}.

\begin{recollection}\label{recollection:simple_perverse_sheaves}
	Let $(X,P)$ be a  stratified space with $P$ finite and consider a perversity $\mathfrak{p} \colon P \to \ZZ$.
	Let $A$ be a discrete regular ring and $p \in P$.
	Assume that the conclusion of \cref{perverse_t_structure}-$(3)$ holds (e.g., if $(X,P)$ is a complex Whitney stratified space).
	Consider the open immersion $j_p \colon X_p \hookrightarrow X_{\leq p}$ and the closed immersion $i_{\leq p} \colon X_{\leq p} \to X$.
	We abuse the notations an denote the composition $P_{\leq P} \subset P \xrightarrow{\mathfrak{p}} \ZZ$ also by $\mathfrak{p}$.
	In the adjunction $j_{p, !} \dashv j_p^\ast \dashv j_{p, \ast}$, the first and third functors are fully faithful.
	Hence we get a natural transformation of functors $\Loc_\omega(X_p; \Mod_A) \to \Cons_{P_{\leq p}, \omega}(X_{\leq p}; \Mod_A) $
\[
j_{p, !} \Longrightarrow j_{p, \ast} \ .
\]
	Moreover, the functor $j_p^\ast$ is ${}^{\mathfrak{p}}\tau$-exact, so that $j_{p, !}$, $j_{p, \ast}$ are respectively right and left ${}^{\mathfrak{p}}\tau$-exact.
	Hence post-composing with the $0$-th perverse cohomology functor, we get a natural transformation of functors $\Loc_\omega(X_p; \Mod_A) \to {}^{\mathfrak{p}}\Perv_{P_{\leq p}, \omega}(X_{\leq p}; \Mod_A)$
\[
{}^{\mathfrak{p}}j_{p, !} \coloneqq {}^{\mathfrak{p}}\pi_0 j_{p, !} \Longrightarrow {}^{\mathfrak{p}}j_{p, \ast} \coloneqq {}^{\mathfrak{p}}\pi_0 j_{p, \ast} \ .
\]
	Consider the functor of abelian categories
\begin{align*}
(i_{\leq p})_\ast (j_p)_{!\ast} \colon & \Loc_\omega(X_p; \Mod_A)^\heartsuit \to {}^{\mathfrak{p}}\Perv_{P, \omega}(X; \Mod_A)\\
& L \mapsto (i_{\leq p})_\ast (j_p)_{!\ast} L \coloneqq (i_{\leq p})_\ast {}^{\mathfrak{p}}\mathrm{Im}({}^{\mathfrak{p}}j_{p, !} L \to  {}^{\mathfrak{p}}j_{p, \ast} L) \ ,
\end{align*}
where the image is taken in the abelian category ${}^{\mathfrak{p}}\Perv_{P_{\leq p}, \omega}(X_{\leq p}; \Mod_A)$.
	We have that $F \in {}^{\mathfrak{p}}\Perv_{P, \omega}(X; \Mod_A)$ is simple if and only if there exists $p \in P$ and a simple local system $L \in \Loc_\omega(X_p; \Mod_A)^\heartsuit$ such that $F \simeq (i_{\leq p})_\ast (j_p)_{!\ast} L$.
\end{recollection}

	The following results will allow us to apply the theory of \cref{subsection:moduli_flat_objects}, together with openness of flatness that we shall discuss in the next paragraph.

\begin{lemma}\label{lemma:Perv_admissible}
	Let $(X,P)$ be a complex Whitney stratified space.
	Let $\mathfrak{p} \colon P \to \ZZ$ be a function and let $A \in \dAff_k$.
	Then the perverse $t$-structure ${}^{\mathfrak{p}}\tau$ on $\Cons_P(X; \Mod_A)$ is admissible.
\end{lemma}

\begin{proof}
	By \cite[Theorem 4.1.22]{Lam25Perv}, this is a special case of \cite[Lemma 4.2.6]{Lam25Perv}.
\end{proof}

\subsection{\texorpdfstring{The stacks $\mathbf{Cons}_P(X)$ and ${}^{\mathfrak{p}}\mathbf{Perv}_P(X)$}{The stacks mathbf{Cons}P(X) and {}{mathfrak{p}}mathbf{Perv}P(X)}}

\begin{recollection}
	Let $(X,P)$ be a complex Whitney stratified space and $\mathfrak{p} \colon X \to P$ be a function.
	The derived prestack of constructible sheaves is defined by the assignment
\begin{align*}
	\mathbf{Cons}_P(X) \colon & \dAff_k \to \mathrm{Spc}\\
	& \Spec (A) \mapsto Cons_{P,\omega}(X; \Mod_A)^{\simeq}
\end{align*}
with functoriality given by extension of scalars. 
	If $P$ is finite, we let
\[
{}^{\mathfrak{p}}\mathbf{Perv}_P(X) \subset \mathbf{Cons}_P(X)
\]
the sub-prestack over $k$ spanned by those constructible sheaves that are ${}^{\mathfrak{p}}\tau$-flat.
	By \cite[Theorem 4.2.11]{HPT26}, the prestack ${}^{\mathfrak{p}}\mathbf{Perv}_P(X)$ is a derived stack.
	We will refer to it as the perverse character stack.
\end{recollection}

	Together with \cref{lemma:Perv_admissible}, the following shows that the stack of perverse sheaves falls in the general framework of \cref{subsection:moduli_flat_objects}.

\begin{theorem}\label{openness_flatness_perv}
    Let $(X, R)$ be a complex Whitney stratified space and $\mathfrak{p} \colon P\to \ZZ$ be a function. 
    Then:
\begin{enumerate}\itemsep=0.2cm
    \item There is a canonical equivalence
\[
\M_{\Cons_P(X; \Mod_k)} \simeq \mathbf{Cons}_P(X) \ .
\]
    In particular $\mathbf{Cons}_P(X)$ is a locally geometric derived stack, locally of finite presentation over $k$.
    \item The perverse $t$-structure ${}^{\mathfrak{p}}\tau$ universally satisfies openness of flatness, that is, the morphism of derived stacks
\[
{}^{\mathfrak{p}}\mathbf{Perv}_P(X) \hookrightarrow \mathbf{Cons}_P(X)
\]
is representable by an open immersion.
    In particular ${}^{\mathfrak{p}}\mathbf{Perv}_P(X)$ is a $1$-Artin stack locally of finite presentation over $k$.
\end{enumerate}
\end{theorem}

\begin{proof}
	By \cite[Theorem 4.1.22]{Lam25Perv}, the equivalence in item $(1)$ is \cite[Lemma 4.1.9]{HPT26}.
	The last part of $(1)$ follows by \cref{repr_moduli_of_objects} and \cref{proposition:Whitney_finite_type}.
	Again by \cite[Theorem 4.1.22]{Lam25Perv}, item $(2)$ is a special case of \cite[Theorem 4.2.11]{HPT26}.
\end{proof}

	In what follows we will need also a deeper understanding of what happens in the unstratified case, that is, for stack of locally constant sheaves.
	
\begin{recollection}[Monodromy equivalence]\label{recollection:monodromy}
	Let $X$ be a connected locally weakly contractible topological space (e.g., a manifold) and let $x \in X$.
	Let $\Pi_\infty(X)$ denote the $\infty$-grupoid of singular chain on $X$.
	By \cite[Observation 3.11]{HPT23} and \cite[Theorem 6.4.4]{PT22}, there are canonical equivalences
\[
\Loc(X; \Mod_k) \simeq \Fun(\Pi_\infty(X), \Mod_k) \simeq \Mod_{C_\bullet(\Omega_{x}X; k)} \ ,
\]
where again we did not differentiate between sheaves and hypersheaves since we in the cases of interest for us the two notions agree.
	The first equivalence, called monodromy equivalence, can be vastly generalized to the case of constructible sheaves and is at the heart of the representability of $\mathbf{Cons}_P(X)$.
	See \cite{CJ24, PT22, HPT24, HPT26}.
\end{recollection}

	In the unstratified setting, we can give an easier description of the stacks of constructible sheaves and of perverse character stacks.
	This also justifies the terminology "perverse character varieties".
	 
\begin{recollection}[Classical character varieties]
	Let $X$ be a connected locally weakly contractible topological space.
	The derived prestack of locally constant sheaves is defined by the assignment
\begin{align*}
	\mathbf{Loc}(X) \colon & \dAff_k \to \mathrm{Spc}\\
	& \Spec (A) \mapsto Loc_{\omega}(X; \Mod_A)^{\simeq}
\end{align*}
with functoriality given by extension of scalars. 
	Consider the substack	
\[
\mathbf{Loc}(X)^{[0,0]} \subset \mathbf{Loc}(X)
\]
the sub-prestack over $k$ spanned by those locally constant sheaves whose stalks are perfect and of Tor-amplitude $[0,0]$.
	If $\Pi_\infty(X)$ is a compact object of $\Cat_\infty$, then $\mathbf{Loc}(X)$ is a locally geometric derived stack, locally of finite presentation over $k$ (e.g., by \cite[Proposition 4.1.12]{HPT26} with $P= \{ \ast \}$).
	The monodromy equivalence identifies $\mathbf{Loc}(X)^{[0,0]}$ with the classical derived character stack
\[
\mathbf{Loc}(X)^{[0,0]} \simeq \bigsqcup_{r \in \N} \mathbf{Map}(\Pi_\infty(X), \mathrm{B}GL_r) \ .
\]
	It follows that the classical truncation $\mathbf{Loc}(X)^\heartsuit$ of $\mathbf{Loc}(X)^{[0,0]}$ is identified with the classical character stack parametrizing representations of the fundamental group of $X$, and its good moduli space with the classically defined character variety of $X$.
\end{recollection}

\subsection{Hochschild homology of constructible sheaves categories}\label{subsection:HO_Cons}

	In this paragraph we give an explicit description of the morphism $\HO \colon \C \to \M_\C$ for $\C \coloneqq \Cons_P(X; \Mod_k)$.
	Recall the notations of \cref{recollection:simple_perverse_sheaves}.

\begin{lemma}\label{lemma:HH_Cons}
	Let $(X,P)$ a complex Whitney stratified space.
	For each $p \in P$, set $i_p \colon X_p \to X$ the canonical inclusion and choose a point $x_p \in X_P$.
	Then we have canonical equivalences
\[
\HH(\Cons_P(X; \Mod_k)) \simeq \bigoplus_{p \in P} \HH(\Loc(X_p; \Mod_k)) \simeq \bigoplus_{p \in P} \HH(\Mod_{C_\bullet(\Omega_{x_p}X_p; k))}) \ ,
\]
where the first equivalence is given by $\oplus_{p \in P} \HH((i_{\leq p})_\# (j_{p,!})$.
\end{lemma}

\begin{proof}
	Recall that, by definition, the poset $P$ is finite.
	The first equivalence follows directly by applying recursively \cref{lemma:splitting_HH}, which is possible thanks to \cref{proposition:recollements_Cons}.
	The second equivalence follows directly by \cref{recollection:monodromy}.
\end{proof}

	\cref{lemma:splitting_HO_description} and \cref{lemma:HH_Cons} yield the following

\begin{corollary}\label{corollary:HO_Cons}
	Let $(X,P)$ a complex Whitney stratified space.
	For each $p \in P$, set $i_p \colon X_p \to X$ the canonical inclusion and choose a point $x_p \in X_P$.
	Under the equivalence of \cref{lemma:HH_Cons}, the map $\HO(\Cons_P(X; \Mod_k)) \to \mathrm{H}^0\mathrm{R}\Gamma(\mathbf{Cons}_P(X); \OO_{\mathbf{Cons}_P(X)})$ is identified with
\begin{align*}
\HO \colon & \bigoplus_{p \in P} \HH(\Mod_{C_\bullet(\Omega_{x_p}X_p; k))}) \to \mathrm{H}^0\mathrm{R}\Gamma(\M_\C, \OO_{\M_\C}) \\
& HO(([\gamma_p])_{p \in P})([x_F]) = \sum_{p \in P} tr(\rho_p(\gamma_p)) \ ,
\end{align*}
where $x_F \colon \Spec(k) \to \mathbf{Cons}_P(X)$ corresponds to $F \in \Cons_{P, \omega}(X; \Mod_k)$, $\gamma_p \in C_0(\Omega_{x_p}X_p; k)$ and $\rho_p(\gamma_p)$ is the quasi-isomorphism of the stalk at $x_p$ of $i_p^\ast F \in \Loc(X_p; \Mod_k)$ induced by the monodromy equivalence \cref{recollection:monodromy}.
\end{corollary}

\begin{remark}
	In the statement of \cref{corollary:HO_Cons} we have been slightly imprecise.
	The elements $\gamma_p \in C_0(\Omega_{x_p}X_p; k)$ are linear combinations of loops in $X_p$ based at $x_p$, hence the expression "$\rho(\gamma_p)$" does not really makes sense.
	Nevertheless its trace does makes sense, defined as the linear combination of the traces of a chosen presentation of $\gamma_P$, which is all we need for the statement to make sense.
\end{remark}

\section{Perverse character varieties}

	In this paragraph we recall the results from \cite{Lam25Perv} about the existence of perverse character varieties.

\subsection{Existence of perverse character varieties}

    Let $k$ be a discrete noetherian ring containing $\mathbb{Q}$ and $k \subset \kappa$ with $\kappa$ an algebraically closed field.
    We have the following

\begin{theorem}[{\cite[Theorem 4.3.33]{Lam25Perv}}]\label{good_Perv_algebraic}
    Let $(X, P)$ be a complex Whitney stratified space and $\mathfrak{p} \colon P \to \ZZ$ the middle perversity.
    The algebraic stack $t_0{}^{\mathfrak{p}}\mathbf{Perv}_P(X)$ admits a separated good moduli space $t_0{}^{\mathfrak{p}}\Perv_P(X)$ whose connected components are of finite presentation over $k$.
	Moreover, the $\kappa$-points of $t_0{}^{\mathfrak{p}}\Perv_P(X)$ parametrize semisimple perverse sheaves with perfect stalks.
	Furthermore, $t_0{}^{\mathfrak{p}}\Perv_P(X)$ admits a natural derived enhancement ${}^{\mathfrak{p}}\Perv_P(X)$ which is a derived good moduli space for ${}^{\mathfrak{p}}\mathbf{Perv}_P(X)$.
\end{theorem}

\begin{notation}
	From now on, when in it exists, we refer to the algebraic space ${}^{\mathfrak{p}}\Perv_P(X)$ as the perverse character variety.
\end{notation}

	\cref{good_Perv_algebraic} can be made more precise.
	Indeed the general results of \cite[Section 3]{Lam25Perv} allow one to prove the existence of good moduli space for moduli of flat objects $\M_\C^\heartsuit$ in great generality.
	In order to apply the results of \textit{loc.cit.}, one needs decompose $t_0{}^{\mathfrak{p}}\Perv_P(X)$ in open and closed substacks that are quasi-compact.
	
\begin{definition}
	Let $X$ be a stratified space and $\chi \colon X \to \ZZ$ be a function.
	We say that $\chi$ is constructible if there exists a stratification $(X,P)$ of $X$ such that $\chi$ is constant on the strata of $P$.
	In this case, we also say that $\chi$ is $P$-constructible.
\end{definition}
	
\begin{definition}\label{definition:EP_index}
	Let $(X, P)$ be a stratified space and $K$ a field.
	For $F \in \Cons_{P, \omega}(X; \Mod_K)$, define the local Euler-Poincaré index of $F$ by 
\begin{align*}
	\chi(F) \colon & X \to \ZZ \\
	& x \mapsto \chi(F_x) \ .
\end{align*}
\end{definition}

\begin{definition}\label{perv_Cons_stack_def}
	Let $(X, P)$ be a stratified space and let $\mathfrak{p} \colon P \to \ZZ$ be a function.
	Let $\chi \colon X \to \ZZ$ be a constructible function.
	We define the sub-presheaf of $\mathbf{Cons}_P(X)$
\[
\mathbf{Cons}_P^{\chi}(X) \colon \dAff_k^{op} \to \mathrm{Spc}
\]
by sending $\Spec(A) \in \dAff_k$ to the maximal $\infty$-subgrupoid of $\Cons_{P, \omega}(X; \Mod_A)$ spanned by those objects $F$ for which the following property holds: for every field $K$ and every $A \to \pi_0(A) \to K$ we have $\chi(F_x \otimes_A K) = \chi(x)$ for every $x \in X$.
\end{definition}

\begin{lemma}\label{lemma:Cons_chi}
	Let $(X,P)$ be a conically stratified space with locally weakly contractible strata.
	Assume that $P$ is finite and that each stratum has a finite number of connected components.
	Then the canonical morphism 
\[
\mathbf{Cons}_P^{\chi}(X) \to \mathbf{Cons}_P(X)
\]
is representable by an open and closed immersion.
\end{lemma}

\begin{proof}
	The same proof of \cite[Lemma 4.3.26]{Lam25Perv} works replacing the perverse character stack with the stack of constructible sheaves.
\end{proof}

\begin{definition}\label{perv_chi_stack_def}
	Let $(X, P)$ be a complex Whitney stratified space and let $\mathfrak{p} \colon P \to \ZZ$ be a function.
	Let $\chi \colon X \to \ZZ$ be a constructible function.
	We define the substack ${}^{\mathfrak{p}}\mathbf{Perv}_P^{\chi}(X) \subset{}^{\mathfrak{p}}\mathbf{Perv}_P(X)$ by the pullback square
\[
\begin{tikzcd}[sep=small]
	{{}^{\mathfrak{p}}\mathbf{Perv}_P^{\chi}(X)} && {{}^{\mathfrak{p}}\mathbf{Perv}_P(X)} \\
	\\
	{\mathbf{Cons}^{\chi}_P(X)} && {\mathbf{Cons}_P(X)}
	\arrow[from=1-1, to=1-3]
	\arrow[from=1-1, to=3-1]
	\arrow[from=1-3, to=3-3]
	\arrow[from=3-1, to=3-3]
\end{tikzcd}
\]
\end{definition}

\begin{remark}
	Unraveling the definitions, the derived stack ${}^{\mathfrak{p}}\mathbf{Perv}_P^{\chi}(X)$ of \cref{perv_chi_stack_def} sends $\Spec(A) \in \dAff_k$ of  to the maximal $\infty$-subgrupoid of $\Cons_{P, \omega}(X; \Mod_A)$ spanned by those objects $F$ for which the following property holds: $F$ is ${}^{\mathfrak{p}}\tau$-flat over $A$ and for every field $K$ and every $A \to \pi_0(A) \to K$ we have $\chi(F_x \otimes_A K) = \chi(x)$ for every $x \in X$.
\end{remark}

	Let $k$ be a ring containing $\mathbb{Q}$ and $k \subset \kappa$ with $\kappa$ an algebraically closed field.
    We have the following

\begin{theorem}\label{theorem:Perv_chi_qc}
	In the setting of \cref{perv_chi_stack_def}, let $\mathfrak{p}\colon X \to P$ be the middle perversity.
	The canonical morphism
\[
{}^{\mathfrak{p}}\mathbf{Perv}_P^{\chi}(X) \to {}^{\mathfrak{p}}\mathbf{Perv}_P(X)
\]
is representable by an open and closed immersion.
	Moreover ${}^{\mathfrak{p}}\mathbf{Perv}_P^{\chi}(X)$ is quasi-compact and admits a derived good moduli space ${}^{\mathfrak{p}}\mathbf{Perv}_P^{\chi}(X)$ whose $\kappa$-points parametrize semisimple perverse sheaves with fixed Euler-Poincaré index $\chi$.
	Moreover the algebraic space $t_0{}^{\mathfrak{p}}\mathbf{Perv}_P^{\chi}(X)$ is of finite presentation over $k$.
\end{theorem}

\begin{proof}
	The first claim follows directly from \cref{lemma:Cons_chi} (see also \cite[Lemma 4.3.26]{Lam25Perv}).
	The second part of the statement follows by \cref{good_Perv_algebraic}.
	For the last claim, notice that ${}^{\mathfrak{p}}\mathbf{Perv}_P(X) \in \dSt_k$ is the basechange of its $\mathbb{Q}$-analogue ${}^{\mathfrak{p}}\mathbf{Perv}_P(X)_\mathbb{Q} \in \dSt_\mathbb{Q}$ along $\mathbb{Q} \subset k$.
	Hence it is enough to prove that the statement holds over $\mathbb{Q}$, which follows by \cref{good_moduli_properties}-$(7)$.
\end{proof}

\begin{remark}\label{remark:good_moduli_Cons}
	When $\mathfrak{p}$ is the trivial perversity of \cref{definition:zero_perversity}, the above theorem holds more generally and the Euler-Poincaré index coincide with the rank of the stalks (see \cite[Section 3]{Lam25Stokes}).
	Notice that the notations in \textit{loc.cit.} are slightly different and reflect the fact that the Euler-Poincaré index of a constructible sheaf in the heart of the standard $t$-structure coincide with the rank of its stalks.
\end{remark}

\begin{notation}\label{notation:good_moduli_Cons}
	In the setting \cref{remark:good_moduli_Cons}, we denote the perverse character stack for the trivial perversity by $\mathbf{Cons}_P(X)^{[0,0]}$ and its good moduli by $\Cons_P(X)^{[0,0]}$, accordingly with \cite{Lam25Stokes}.
	The underlying classical algebraic stacks are denoted respectively by $\mathbf{Cons}_P(X)^\heartsuit$ and $\Cons_P(X)^\heartsuit$.
\end{notation}

\subsection{Quasi-affineness of perverse character varieties}\label{subsection:quasi_affineness}
	
	In this paragraph we prove the main result of the present paper.
	More precisely, we show that the sections provided by the morphism $\HO \colon \HH(\C) \to \mathrm{H}^0\mathrm{R}\Gamma(\M_\C, \OO_{\M_\C})$ are enough to separate points of ${}^{\mathfrak{p}}\mathbf{Perv}_P(X)$, for any perversity $\mathfrak{p}$.
	This yields the quasi-affineness of perverse character varieties.\\ \indent
	We will need a classical result about representations of finitely presented groups.
    
\begin{definition}
	Let $\Gamma$ be a group, $k$ a field and $V$ a finite-dimensional $k$-vector space.
	Let $\rho \colon G \to \mathrm{GL}(V)$ be a representation.
	The character associated to $\rho$ is the function
\begin{align*}
tr_\rho \colon &G \to k\\
& g \mapsto tr(\rho(g)) \ ,
\end{align*}
where $tr$ is the trace.
\end{definition}

	We will need the following classical result.
\begin{lemma}\label{lemma:characters_separate_points}
	Let $\Gamma$ be a finitely generated group and $\kappa$ an algebraically closed field of characteristic zero.
	Let $\rho_1, \rho_2$ be finite-dimensional semisimple $\kappa$-representations of $\Gamma$.
	Then $\rho_1$ and $\rho_2$ are isomorphic if and only if their characters agree.
\end{lemma}

\begin{corollary}\label{proposition:separating_Loc}
	Let $X$ be a connected locally weakly contractible topological space and let $\kappa$ be an algebraically closed field.
	Let $L_1,L_2 \in \Loc(X; \Mod_\kappa)^\heartsuit$ semisimple of finite rank with $L_1 \not\simeq L_2$.
	Assume that $\Pi_\infty(X)$ is a compact object in $\Cat_\infty$.
	Then there is a global section in the image of
\[
\HO_X \colon \HH_0(\Loc(X); \Mod_k) \to \mathrm{H}^0\mathrm{R}\Gamma(\mathbf{Loc}(X), \OO_{\mathbf{Loc}(X)})
\]
that separate $L_1$ from $L_2$.
\end{corollary}

\begin{proof}
	Let $x \in X$.
	Since $\Pi_\infty(X)$ is a compact object in $\Cat_\infty$, the fundamental group $\pi_1(X,x)$ of $X$ is finitely presented.
	Via the monodromy equivalence, the locally constant sheaves $L_1, L_2$ correspond to semisimple representations $\rho_1, \rho_2$ of the fundamental group.
	Hence by \cref{lemma:characters_separate_points}, we can choose $\gamma \in \pi_1(X,x)$ such that $tr(\rho_1(\gamma)) \neq tr(\rho_2(\gamma))$.
	Up to choosing a representative, we can consider $\gamma$ as an element of $C_0(\Omega_x X; \Mod_\kappa)$.
	Recall that if the ranks of the representations $\rho_1$ and $\rho_2$ are different, then they lie in different connected component.
    In particular, the representations $\rho_1$ and $\rho_2$ are separated by image of the constant loop $c_x$ at $x$.
    Indeed, by \cref{corollary:HO_Cons} in the case $P = \ast$, we get that $\HO_X(c_x)$ represents the trace of the identity, i.e., the rank of a representation.
	So we can assume that $L_1$ and $L_2$ have the same rank and we denote it by $r$.
    Since $L_1 \not\simeq L_2$, by \cref{lemma:characters_separate_points} there exists $\gamma \in \pi_1(X,x)$ such that $tr(\rho_1(\gamma)) \neq tr(\rho_2(\gamma))$.
	Again by \cref{corollary:HO_Cons}, we get that
\[
\HO_X\left(\gamma - \frac{tr(\rho_2(\gamma))}{\dim_\kappa(L_2)_x}c_x\right) \in \mathrm{H}^0(\mathbf{Loc}(X), \OO_{\mathbf{Loc}(X)})
\]
is zero on the point representing $L_2$ and non-zero on the point representing $L_1$.
\end{proof}

Recall the description of simple perverse sheaves from \cref{recollection:simple_perverse_sheaves}.
	We will need the following easy lemma:

\begin{lemma}\label{lemma:maximal_elements_Perv}
	Under the assumptions of \cref{lemma:Cons_chi}, let $\kappa$ be a field over $k$ and $\mathfrak{p} \colon P \to \ZZ$ a perversity.
	Consider
\[
x_F,y_G \colon \Spec(\kappa) \to t_0{}^{\mathfrak{p}}\mathbf{Perv}_P(X) \subset \mathbf{Cons}_P(X)
\]
representing semisimple perverse sheaves $F, G \in {}^{\mathfrak{p}}\Perv_{P, \omega}(X; \Mod_\kappa)$.
	Assume that the conclusion of \cref{perverse_t_structure}-$(3)$ holds for $(X,P)$ and $A = \kappa$ and write
\[
F \coloneqq \bigoplus_{p \in P} (i_{\leq p})_\ast (j_{p})_{!\ast} F_{p} \ , \qquad G \coloneqq \bigoplus_{p\in P} (i_{\leq p})_\ast (j_{p})_{!\ast} G_{p} \ ,
\]
where $F_{p} \in \Loc_\omega(X_{p}; \Mod_\kappa)^\heartsuit$ and $G_{p} \in Loc_\omega(X_{p}; \Mod_\kappa)^\heartsuit$ are non-zero semisimple locally constant sheaves of finite rank. Let $P_F$ (respectively $P_G$) be the subset $\,\{p \in P \,\vert F_p\neq 0\} \subset P$ (respectively $\,\{p \in P \,\vert G_p\neq 0\} \subset P$) and let $q$ be a maximal element in $P_F$.
If $x$ and $y$ lie in the same connected component of $\mathbf{Cons}_P(X)$, then:
\begin{enumerate}\itemsep=0.2cm
	\item $q \in P_G$ and it is a maximal element in $P_G$;
	\item $\dim_\kappa (F_{q})_x = \dim_\kappa (G_{q})_x$ for each for any $x \in X_{q}$.
\end{enumerate}	 
\end{lemma}

\begin{proof}
	If $q$ is not a maximal element of the list $I \coloneqq \{q \} \cup P_G \subset P$, then $F$ and $G$ lie in different connected components.
	Indeed, if it is not the case, consider a maximal element of $q' \in I$ and consider an element $x_{q'} \in X_{q'}$.
	Then we get
\[
\chi(F)(x_{q'}) = 0 \neq \dim_\kappa G_{q'} = \chi(G)({x_{q'}}) \ .
\]
	Hence $q$ must be a maximal element of $I$ by \cref{lemma:Cons_chi}.
	If $q \notin I$, by replacing the roles of $F$ and $G$ above one sees that $x$ and $y$ do not lie in the same connected component.
	This proves $(1)$.
	Item $(2)$ follows since $q$ is maximal and if $x,y$ lie in the same connected component, then by \cref{lemma:Cons_chi} we must have
\[
\dim_\kappa (F_{q})_x = \chi(F)(x_{q}) = \chi(G)({x_{q}}) = \dim_\kappa (G_{q})_x \ .
\]
	This proves $(2)$.
\end{proof}

	We are now ready to prove our main result:
\begin{theorem}\label{theorem:separating_points_Perv}
	Under the assumptions of \cref{lemma:maximal_elements_Perv}, let $\kappa$ be of characteristic zero.
	Assume that $\Pi_\infty(X_p)$ is a compact object in $\Cat_\infty$ for each $p \in P$.
	Then the structure sheaf $\OO_{t_0\mathbf{Cons}_P(X)}$ separates points representing semisimple perverse sheaves of ${}^{\mathfrak{p}}\Perv_{P, \omega}(X; \Mod_\kappa)$.
	More precisely, such points can be separated by global sections in the image of the morphism
\[
\HO_X \colon \HH_0(\Cons_P(X; \Mod_k)) \to \mathrm{H}^0\mathrm{R}\Gamma(\mathbf{Cons}_P(X), \OO_{\mathbf{Cons}_P(X)})  \ .
\]
\end{theorem}

\begin{proof}
	Let $x_F, x_G \colon \Spec(\kappa) \to \mathbf{Cons}_P(X)$ parametrizing semisimple perverse sheaves $F,G \in {}^{\mathfrak{p}}\Perv_{P, \omega}(X; \Mod_\kappa)$.
	By \cref{recollection:simple_perverse_sheaves}, there exist semisimple local systems of finite rank $F_{p}, G_{p} \in Loc_\omega(X_{p}; \Mod_\kappa)^\heartsuit$ such that
\[
F \simeq \bigoplus_{p\in P} (i_{\leq p})_\ast (j_{p})_{!\ast} F_{p} \ , \qquad G \simeq \bigoplus_{p\in P} (i_{\leq p})_\ast (j_{p})_{!\ast} G_{p} \ .
\]
	Following the notation of \Cref{lemma:maximal_elements_Perv}, let $P_F$ (respectively $P_G$) be the subset $\,\{p \in P \,\vert F_p\neq 0\} \subset P$ (respectively $\,\{p \in P \,\vert G_p\neq 0\} \subset P$) and let $n$ (respectively $m$) be the cardinality of $P_F$ (respectively $P_G$). We proceed by induction on $n$. \\ \indent
	Suppose that $n=1$, and thus $P_F=\{q\}$.
	By \cref{lemma:maximal_elements_Perv}, we have that $q \in P_G$ is a maximal element of $P_G$. We also have that $F_{q}$ and $G_{q}$ are semisimple locally constant sheaves of the same rank. If $F_{q}$ and $G_{q}$ are not isomorphic in $\Loc_\omega(X_q; \Mod_\kappa)^\heartsuit$, then we are done by \cref{corollary:HO_Cons} and \cref{proposition:separating_Loc}.
	Indeed by \cite[Theorem 5.1.7-(5)]{HPT24} and \cite[Theorem 4.1.22]{Lam25Perv}, the assumptions of \cref{proposition:separating_Loc} are satisfied.
	Hence the points corresponding to $F_q$ and $G_q$ can be separated in $\mathbf{Loc}(X_q)$ by global sections in the image of
\[
\HO_{X_q} \colon \HH_0(\Loc(X_q); \Mod_k) \to \mathrm{H}^0(\mathbf{Loc}(X_q), \OO_{\mathbf{Loc}(X_q)}) \, .
\]
	By the splitting provided by \cref{corollary:HO_Cons}, it follows that $x_F,y_F \in \mathbf{Cons}_P(X)$ can be separated by sections in the image of $\HO_X$.
	Hence we are reduced to consider the case $F_{q} \simeq G_{q}$.
	We analyze separately the cases $m=1$ and $m>1$.
	If $m=1$, then $F \simeq G$ and thus the points $x$ and $y$ agree.
	In this case there is nothing to prove.
	If $m>1$, since the traces of \cref{corollary:HO_Cons} are additive in direct sums (more generally, in short exact sequences), then $x$ and $y$ lie in different connected components by \cref{lemma:Cons_chi}.
	To see this, choose a maximal $q' \in P_G\setminus \{q\}$ and consider a point of $x' \in X_{q'}$.
	Since the Euler-Poincaré index is additive in short exact sequences, we get
\begin{align*}
\chi(F)(x') & = \chi((i_{\leq q})_\ast (j_{q})_{!\ast} F_{q})(x') = \chi((i_{\leq q})_\ast (j_{q})_{!\ast} G_{q})(x') \\
& \neq \chi((i_{\leq q})_\ast (j_{q})_{!\ast} G_{q})(x') + \dim_\kappa (G_{{q'}})_{x'} = \chi(G)(x') \ .
\end{align*}
	Hence if $m>1$ there is nothing to prove. 
	We now proceed by induction, supposing the result is true for $n-1$.
	Let $q \in P_F$ be a maximal element. Again by \cref{lemma:maximal_elements_Perv}, we have that $q \in P_G$ and that $q$ is a maximal element of $P_G$.
	Arguing as above, we reduce to treat the case $F_{q} \simeq G_{q}$.
	Since the traces of \cref{corollary:HO_Cons} are additive in direct sums (more generally, in short exact sequences), we have that the structure sheaf $\OO_{t_0{}^{\mathfrak{p}}\mathbf{Perv}_P(X)}$ separates $x,y$ if and only if it separates the point representing
\[
F' \coloneqq \bigoplus_{p \in P\setminus \{q\}} (i_{\leq p})_\ast (j_{p})_{!\ast} F'_{p} \ , \qquad G' \coloneqq \bigoplus_{p \in P\setminus \{q\}} (i_{\leq p})_\ast (j_{p)_{!\ast}} G_{p}' \ ,
\]
where $F'_p$ (respectively $G'_p$) is set to be equal to $F_p$ (respectively $G_p$) if $p\neq q$, whereas we set $F'_p=0$ (respectively $G'_q=0$).
The result now follows by our inductive hypothesis, since $P_{F'}=P_F\setminus \{q\}$.
\end{proof}

\begin{corollary}\label{corollary:quasi_affine_Perv}
	Let $(X,P)$ be a complex Whitney stratified space and let $\mathfrak{p} \colon X \to \ZZ$ be the middle perversity.
	Let $k$ be a noetherian ring containing $\mathbb{Q}$.
	Then the perverse character variety $t_0 {}^{\mathfrak{p}}\Perv_P(X)$ is a disjoint union of quasi-affine scheme over $k$.
\end{corollary}

\begin{proof}
	By \cref{good_Perv_algebraic} the polystable points of $t_0 {}^{\mathfrak{p}}\Perv_P(X)$ are the semisimple perverse sheaves.
	By \cref{openness_flatness_perv} and \cref{perverse_t_structure}, the assumptions of \cref{theorem:separating_points_Perv} are satisfied.
	It then follows directly from the definition of good moduli space and \cref{good_moduli_properties}-$(4)$ that $\OO_{t_0 {}^{\mathfrak{p}}\Perv_P(X)}$ separates points.
	The claim follows by applying \cref{corollary:quasi_affineness_good} to each connected components of $t_0 {}^{\mathfrak{p}}\Perv_P(X)$, which is possible thanks to \cref{theorem:Perv_chi_qc}.
\end{proof}

	In \cref{corollary:quasi_affine_Perv} we have shown the quasi-affineness in the important case of the middle perversity.
	Our results are completely general and apply to any perversity as long as the corresponding perverse character variety exists.
	Another example is given by the trivial perversity introduced in \cref{definition:zero_perversity}, whose associated perverse $t$-structure is the standard $t$-structure, yielding a stack $\mathbf{Cons}_P(X)^{[0,0]}$  parametrizing constructible sheaves with flat stalks (see \cref{remark:trivial_perversity} and \cref{notation:good_moduli_Cons}).
	Recall the notions of (local) categorical compactness for a stratified space $(X,P)$ from \cite[Definition 2.2.1]{PT22}).

\begin{corollary}\label{corollary:quasi_affine_Cons}
	Let $(X,P)$ be a categorically compact, locally categorically compact conically stratified space with locally weakly contractible strata and $P$ finite.
	Let $k$ be a noetherian ring containing $\mathbb{Q}$.
	The constructible character variety $\Cons_P(X)^\heartsuit$, that is, the good moduli space of $\mathbf{Cons}_P(X)^{[0,0]}$ is a disjoint union of quasi-affine schemes over $k$.
\end{corollary}

\begin{proof}
	By \cite[Proposition 4.1.12]{HPT26}, we have that $\mathbf{Cons}_P(X)^{[0,0]}$ is representable by an algebraic stack.
	By \cite[Theorem 3.2.16]{Lam25Stokes}, its truncation admits a good moduli space $\Cons_P(X)^\heartsuit$.
	That \cref{perverse_t_structure}-$(3)$ holds for $(X,P)$ and $k$ is a special case of \cite[Proposition 4.2.5-(3)]{Lam25Perv} whose polystable points are in bijection with semisimple constructible sheaves.
	Hence the assumptions of \cref{theorem:separating_points_Perv} are satisfied.
	It follows directly from the definition of good moduli space, by \cref{good_moduli_properties}-$(4)$ that $\OO_{t_0 Perv_P(X)}$ separates points.
	The result then follows by applying \cref{corollary:quasi_affineness_good} to each connected components of $\Cons_P(X)^\heartsuit$, which is possible by \cite[Proposition 3.1.10 \& Notation 3.2.15]{Lam25Stokes}.
\end{proof}

\begin{example}
	\cref{corollary:quasi_affine_Cons} applies for example to complex Whitney stratified spaces.
	It also applies to a wide variety of conical algebraic and subanalytic stratifications.
	See \cite{HPT24, HPT26}.
\end{example}

\bibliographystyle{alpha}
\bibliography{bibliography}

\end{document}

%% file: newcommand.tex
\ifpersonal
\newcommand{\personal}[1]{\textcolor[rgb]{0,0,1}{(Personal: #1)}}
\newcommand{\discussion}[1]{\textcolor{violet}{(Discussion: #1)}}
\else
\newcommand{\personal}[1]{\ignorespaces}
\newcommand{\discussion}[1]{\ignorespaces}
\fi

\def\N{\mathbb{N}}

\def\S{\mathbb{S}}
\def\ZZ{\mathbb{Z}}

\def\C{\mathcal{C}}
\def\D{\mathcal{D}}
\def\E{\mathcal{E}}

\def\X{\mathcal{X}}

\def\Z{\mathcal{Z}}
\def\U{\mathcal{U}}
\def\M{\mathcal{M}}
\def\OO{\mathcal{O}}
\def\S{\textsf{S}}

\def\Prlo{\mathrm{Pr}^{\mathrm{L}, \omega}}

\def\Funst{\mathrm{Fun}^{\mathrm{st}}}
\def\FunR{\mathrm{Fun}^{\mathrm{R}}}

\DeclareMathOperator{\QCoh}{QCoh}
\DeclareMathOperator{\dAff}{dAff}
\DeclareMathOperator{\dSt}{dSt}

\DeclareMathOperator{\fib}{fib}
\DeclareMathOperator{\cof}{cof}
\DeclareMathOperator{\Loc}{Loc}

\DeclareMathOperator{\Hom}{Hom}
\DeclareMathOperator{\Cons}{Cons}

\DeclareMathOperator{\Fun}{Fun}
\DeclareMathOperator{\Sh}{Sh}

\DeclareMathOperator{\HH}{HH}

\DeclareMathOperator{\Mod}{Mod}
\DeclareMathOperator{\LMod}{LMod}

\DeclareMathOperator{\Map}{Map}

\DeclareMathOperator{\End}{End}
\DeclareMathOperator{\Cat}{\mathrm{Cat}}
\DeclareMathOperator{\Spc}{Spc}
\DeclareMathOperator{\ev}{ev}
\DeclareMathOperator{\coev}{coev}
\DeclareMathOperator{\Id}{Id}

\DeclareMathOperator{\Perf}{Perf}
\DeclareMathOperator{\Perv}{Perv}
\DeclareMathOperator{\Spec}{Spec}

\DeclareMathOperator{\HO}{HO}

\makeatletter
\renewenvironment{proof}[1][\relax]{\par
  \pushQED{\qed}%
  \normalfont \topsep6\p@\@plus6\p@\relax
  \trivlist
  \item[\hskip\labelsep\itshape
    \ifx#1\relax \proofname\else\proofname{} of #1\fi\@addpunct{.}]\ignorespaces
}{%
  \popQED\endtrivlist\@endpefalse
}
\makeatother

\makeatletter
\let\save@mathaccent\mathaccent
\newcommand*\if@single[3]{%
	\setbox0\hbox{${\mathaccent"0362{#1}}^H$}%
	\setbox2\hbox{${\mathaccent"0362{\kern0pt#1}}^H$}%
	\ifdim\ht0=\ht2 #3\else #2\fi
}
\newcommand*\rel@kern[1]{\kern#1\dimexpr\macc@kerna}
\newcommand*\widebar[1]{\@ifnextchar^{{\wide@bar{#1}{0}}}{\wide@bar{#1}{1}}}
\newcommand*\wide@bar[2]{\if@single{#1}{\wide@bar@{#1}{#2}{1}}{\wide@bar@{#1}{#2}{2}}}
\newcommand*\wide@bar@[3]{%
	\begingroup
	\def\mathaccent##1##2{%
		\let\mathaccent\save@mathaccent
		\if#32 \let\macc@nucleus\first@char \fi
		\setbox\z@\hbox{$\macc@style{\macc@nucleus}_{}$}%
		\setbox\tw@\hbox{$\macc@style{\macc@nucleus}{}_{}$}%
		\dimen@\wd\tw@
		\advance\dimen@-\wd\z@
		\divide\dimen@ 3
		\@tempdima\wd\tw@
		\advance\@tempdima-\scriptspace
		\divide\@tempdima 10
		\advance\dimen@-\@tempdima
		\ifdim\dimen@>\z@ \dimen@0pt\fi
		\rel@kern{0.6}\kern-\dimen@
		\if#31
		\overline{\rel@kern{-0.6}\kern\dimen@\macc@nucleus\rel@kern{0.4}\kern\dimen@}%
		\advance\dimen@0.4\dimexpr\macc@kerna
		\let\final@kern#2%
		\ifdim\dimen@<\z@ \let\final@kern1\fi
		\if\final@kern1 \kern-\dimen@\fi
		\else
		\overline{\rel@kern{-0.6}\kern\dimen@#1}%
		\fi
	}%
	\macc@depth\@ne
	\let\math@bgroup\@empty \let\math@egroup\macc@set@skewchar
	\mathsurround\z@ \frozen@everymath{\mathgroup\macc@group\relax}%
	\macc@set@skewchar\relax
	\let\mathaccentV\macc@nested@a
	\if#31
	\macc@nested@a\relax111{#1}%
	\else
	\def\gobble@till@marker##1\endmarker{}%
	\futurelet\first@char\gobble@till@marker#1\endmarker
	\ifcat\noexpand\first@char A\else
	\def\first@char{}%
	\fi
	\macc@nested@a\relax111{\first@char}%
	\fi
	\endgroup
}
\makeatother

\usetikzlibrary{decorations.markings} 
\tikzset{
  closed/.style = {decoration = {markings, mark = at position 0.5 with { \node[transform shape, xscale = .8, yscale=.4] {/}; } }, postaction = {decorate} },
  open/.style = {decoration = {markings, mark = at position 0.5 with { \node[transform shape, scale = .7] {$\circ$}; } }, postaction = {decorate} }
}

\makeatletter
\tikzcdset{
  open/.code     = {\tikzcdset{hook, circled};},
  closed/.code   = {\tikzcdset{hook, slashed};},
  open'/.code    = {\tikzcdset{hook', circled};},
  closed'/.code  = {\tikzcdset{hook', slashed};},
  circled/.code  = {\tikzcdset{markwith = {\draw (0,0) circle (.375ex);}};},
  slashed/.code  = {\tikzcdset{markwith = {\draw[-] (-.4ex,-.4ex) -- (.4ex,.4ex);}};},
  markwith/.code ={
    \pgfutil@ifundefined%
    {tikz@library@decorations.markings@loaded}%
    {\pgfutil@packageerror{tikz-cd}{You need to say %
      \string\usetikzlibrary{decorations.markings} to use arrows with markings}{}}{}%
    \pgfkeysalso{/tikz/postaction = {
      /tikz/decorate,
      /tikz/decoration={markings, mark = at position 0.5 with {#1}}}
    }
  },
}
\makeatother

\newcommand{\leftrarrows}{\mathrel{\raise.75ex\hbox{\oalign{%
  $\scriptstyle\leftarrow$\cr
  \vrule width0pt height.5ex$\hfil\scriptstyle\relbar$\cr}}}}
\newcommand{\lrightarrows}{\mathrel{\raise.75ex\hbox{\oalign{%
  $\scriptstyle\relbar$\hfil\cr
  $\scriptstyle\vrule width0pt height.5ex\smash\rightarrow$\cr}}}}
\newcommand{\Rrelbar}{\mathrel{\raise.75ex\hbox{\oalign{%
  $\scriptstyle\relbar$\cr
  \vrule width0pt height.5ex$\scriptstyle\relbar$}}}}

\makeatletter
\def\leftrightarrowsfill@{\arrowfill@\leftrarrows\Rrelbar\lrightarrows}
\newcommand{\xleftrightarrows}[2][]{\ext@arrow 3399\leftrightarrowsfill@{#1}{#2}}
\makeatother

\NewCommandCopy{\notocsection}{\section}
\xpatchcmd{\notocsection}{{1}}{{1001}}{}{}

%% file: Quasi_Affineness.bbl
\newcommand{\etalchar}[1]{$^{#1}$}
\begin{thebibliography}{BDN{\etalchar{+}}25}

\bibitem[AG14]{AG14}
Benjamin Antieau and David Gepner.
\newblock Brauer groups and {\'e}tale cohomology in derived algebraic geometry.
\newblock {\em Geom. Topol.}, 18(2):1149--1244, 2014.

\bibitem[AHL26]{AlperHalLei}
Jarod Alper and Daniel Halpern-Leistner.
\newblock {The intrinsic approach to moduli theory}, 2026.

\bibitem[AHLH23]{AHLH23}
Jarod Alper, Daniel Halpern-Leistner, and Jochen Heinloth.
\newblock Existence of moduli spaces for algebraic stacks.
\newblock {\em Inventiones mathematicae}, 234(3):949–1038, August 2023.

\bibitem[AHPS23]{AHPS23}
Eric Ahlqvist, Jeroen Hekking, Michele Pernice, and Michail Savvas.
\newblock Good {Moduli} {Spaces} in {Derived} {Algebraic} {Geometry}.
\newblock Preprint, {arXiv}:2309.16574 [math.{AG}] (2023), 2023.

\bibitem[AHR25]{AHR25}
Jarod Alper, Jack Hall, and David Rydh.
\newblock The \'etale local structure of algebraic stacks.
\newblock Preprint, {arXiv}:1912.06162 [math.{AG}] (2025), 2025.

\bibitem[Alp13]{Alp13}
Jarod Alper.
\newblock Good moduli spaces for {Artin} stacks.
\newblock {\em Ann. Inst. Fourier}, 63(6):2349--2402, 2013.

\bibitem[AOV08]{AOV08}
Dan Abramovich, Martin Olsson, and Angelo Vistoli.
\newblock Tame stacks in positive characteristic.
\newblock {\em Ann. Inst. Fourier}, 58(4):1057--1091, 2008.

\bibitem[Art69]{Art69}
Michael Artin.
\newblock On {Azumaya} algebras and finite dimensional representations of
  rings.
\newblock {\em J. Algebra}, 11:532--563, 1969.

\bibitem[BBDG18]{BBDG18}
Alexander Beilinson, Joseph Bernstein, Pierre Deligne, and Ofer Gabber.
\newblock {\em Faisceaux pervers. {Actes} du colloque ``{Analyse} et
  {Topologie} sur les {Espaces} {Singuliers}''. {Partie} {I}}, volume 100 of
  {\em Ast{\'e}risque}.
\newblock Paris: Soci{\'e}t{\'e} Math{\'e}matique de France (SMF), 2nd edition
  edition, 2018.

\bibitem[BCS24]{BCS24b}
Tristan Bozec, Damien Calaque, and Sarah Scherotzke.
\newblock Relative critical loci and quiver moduli.
\newblock {\em Ann. Sci. {\'E}c. Norm. Sup{\'e}r. (4)}, 57(2):553--614, 2024.

\bibitem[BD21]{BD21}
Christopher Brav and Tobias Dyckerhoff.
\newblock Relative {Calabi}-{Yau} structures. {II}: {Shifted} {Lagrangians} in
  the moduli of objects.
\newblock {\em Sel. Math., New Ser.}, 27(4):45, 2021.
\newblock Id/No 63.

\bibitem[BDN{\etalchar{+}}25]{BDN+25}
Chenjing Bu, Ben Davison, Andr{\'e}s~Ib{\'a}{\~n}ez N{\'u}{\~n}ez, Tasuki
  Kinjo, and Tudor P{\u{a}}durariu.
\newblock Cohomology of symmetric stacks.
\newblock Preprint, {arXiv}:2502.04253 [math.{AG}] (2025), 2025.

\bibitem[BH26]{BH26}
Qingyuan Bai and Peter~J. Haine.
\newblock Localizing invariants of constructible sheaves.
\newblock Preprint, {arXiv}:2512.12810 [math.{KT}] (2026), 2026.

\bibitem[CJ24]{CJ24}
Dustin Clausen and Mikala~{\O}rsnes Jansen.
\newblock The reductive {Borel}-{Serre} compactification as a model for
  unstable algebraic {K}-theory.
\newblock {\em Sel. Math., New Ser.}, 30(1):93, 2024.
\newblock Id/No 10.

\bibitem[CL26a]{CL26b}
Roger Casals and Wenyuan Li.
\newblock Positive microlocal holonomies are globally regular.
\newblock Preprint, {arXiv}:2409.07435 [math.{SG}] (2026), 2026.

\bibitem[CL26b]{CL26}
Merlin Christ and Enrico Lampetti.
\newblock Lagrangian structures on the derived moduli of constructible sheaves.
\newblock Preprint, {arXiv}:2603.04983 [math.{AG}] (2026), 2026.

\bibitem[Dav24]{Dav24}
Ben Davison.
\newblock Purity and 2-{Calabi}-{Yau} categories.
\newblock {\em Invent. Math.}, 238(1):69--173, 2024.

\bibitem[Del70]{Del70}
Pierre Deligne.
\newblock {\em Equations diff{\'e}rentielles {\`a} points singuliers
  r{\'e}guliers}, volume 163 of {\em Lect. Notes Math.}
\newblock Springer, Cham, 1970.

\bibitem[DHM22]{DHM22}
Ben Davison, Lucien Hennecart, and Sebastian~Schlegel Mejia.
\newblock {BPS} {Lie} algebras for totally negative 2-{Calabi}-{Yau} categories
  and nonabelian {Hodge} theory for stacks.
\newblock Preprint, {arXiv}:2212.07668 [math.{RT}] (2022), 2022.

\bibitem[DPS25]{DPS25}
Duiliu-Emanuel Diaconescu, Mauro Porta, and Francesco Sala.
\newblock Cohomological {Hall} algebras, their categorification, and their
  representations via torsion pairs.
\newblock Preprint, {arXiv}:2207.08926 [math.{AG}] (2025), 2025.

\bibitem[GM88]{GM88}
Mark Goresky and Robert MacPherson.
\newblock {\em Stratified {Morse} theory}, volume~14 of {\em Ergeb. Math.
  Grenzgeb., 3. Folge}.
\newblock Berlin etc.: Springer-Verlag, 1988.

\bibitem[Gol84]{Gol84}
William~M. Goldman.
\newblock The symplectic nature of fundamental groups of surfaces.
\newblock {\em Adv. Math.}, 54:200--225, 1984.

\bibitem[Hen24]{Hen24}
Lucien Hennecart.
\newblock Cohomological integrality for symmetric quotient stacks.
\newblock Preprint, {arXiv}:2408.15786 [math.{AG}] (2024), 2024.

\bibitem[HK25]{HK25}
Lucien Hennecart and Tasuki Kinjo.
\newblock The {BPS} decomposition theorem.
\newblock Preprint, {arXiv}:2509.21298 [math.{AG}] (2025), 2025.

\bibitem[Hoy16]{MO:168526}
Marc Hoyois.
\newblock Answer to {M}ath {O}verflow {Q}uestion 168526: Is the
  {$\infty$}-topos {$\Sh(X)$} hypercomplete whenever {$X$} is a {CW} complex?
\newblock
  \href{https://mathoverflow.net/questions/168526/}{\nolinkurl{MO:168526}},
  October 2016.

\bibitem[HPT23]{HPT23}
Peter~J. Haine, Mauro Porta, and Jean-Baptiste Teyssier.
\newblock The homotopy-invariance of constructible sheaves.
\newblock {\em Homology Homotopy Appl.}, 25(2):97--128, 2023.

\bibitem[HPT24]{HPT24}
Peter~J. Haine, Mauro Porta, and Jean-Baptiste Teyssier.
\newblock Exodromy beyond conicality.
\newblock Preprint, available at the authors' webpage
  \url{https://peterjhaine.github.io/#research}, 2024.

\bibitem[HPT26]{HPT26}
Peter~J. Haine, Mauro Porta, and Jean-Baptiste Teyssier.
\newblock The derived moduli of perverse sheaves.
\newblock Preprint, available at the authors' webpage
  \url{https://peterjhaine.github.io/#research}, 2026.

\bibitem[HSS17]{HSS17}
Marc Hoyois, Sarah Scherotzke, and Nicol{\`o} Sibilla.
\newblock Higher traces, noncommutative motives, and the categorified {Chern}
  character.
\newblock {\em Adv. Math.}, 309:97--154, 2017.

\bibitem[Kas84]{Kas84}
Masaki Kashiwara.
\newblock The {Riemann}-{Hilbert} problem for holonomic systems.
\newblock {\em Publ. Res. Inst. Math. Sci.}, 20:319--365, 1984.

\bibitem[KM97]{KM97}
Se{\'a}n Keel and Shigefumi Mori.
\newblock Quotients by groupoids.
\newblock {\em Ann. Math. (2)}, 145(1):193--213, 1997.

\bibitem[KS90]{KS90}
Masaki Kashiwara and Pierre Schapira.
\newblock {\em Sheaves on manifolds. {With} a short history ``{Les} d{\'e}buts
  de la th{\'e}orie des faisceaux'' by {Christian} {Houzel}}, volume 292 of
  {\em Grundlehren Math. Wiss.}
\newblock Berlin etc.: Springer-Verlag, 1990.

\bibitem[Lam25a]{Lam25Perv}
Enrico Lampetti.
\newblock Good moduli for moduli of objects.
\newblock Preprint, available at the authors' webpage
  \url{https://sites.google.com/view/enricolampettimath}, 2025.

\bibitem[Lam25b]{Lam25Stokes}
Enrico Lampetti.
\newblock Good moduli space for constructible sheaves and {Stokes} functors.
\newblock Preprint, available at the authors' webpage
  \url{https://sites.google.com/view/enricolampettimath}, 2025.

\bibitem[LM85]{LM85}
Alexander Lubotzky and Andy~R. Magid.
\newblock {\em Varieties of representations of finitely generated groups},
  volume 336 of {\em Mem. Am. Math. Soc.}
\newblock Providence, RI: American Mathematical Society (AMS), 1985.

\bibitem[Lur17]{Lur17}
Jacob Lurie.
\newblock Higher algebra.
\newblock Available at the authors' webpage
  \url{https://www.math.ias.edu/~lurie/}, 2017.

\bibitem[Lur18]{Lur18}
Jacob Lurie.
\newblock Spectral algebraic geometry.
\newblock Available at the authors' webpage
  \url{https://www.math.ias.edu/~lurie/}, 2018.

\bibitem[Meb84]{Meb84}
Zoghman Mebkhout.
\newblock Une {\'e}quivalence de cat{\'e}gories. {Une} autre {\'e}quivalence de
  cat{\'e}gories.
\newblock {\em Compos. Math.}, 51:51--62, 63--88, 1984.

\bibitem[MFK94]{Mum94}
David~B. Mumford, John~C. Fogarty, and Frances~C. Kirwan.
\newblock {\em Geometric invariant theory.}, volume~34 of {\em Ergeb. Math.
  Grenzgeb.}
\newblock Berlin: Springer-Verlag, 3rd enl. ed. edition, 1994.

\bibitem[Por25]{Por25}
Mauro Porta.
\newblock {\em Derived methods in moduli theory}.
\newblock Habilitation {\`a} diriger des recherches, {Universit{\'e} de
  Strasbourg}, 2025.

\bibitem[Pro75]{Pro75}
Claudio Procesi.
\newblock Finite dimensional representations of algebras.
\newblock {\em Isr. J. Math.}, 19:169--182, 1975.

\bibitem[PT22]{PT22}
Mauro Porta and Jean-Baptiste Teyssier.
\newblock Topological exodromy with coefficients.
\newblock Preprint, available at the authors' webpage
  \url{http://jbteyssier.com/}, 2022.

\bibitem[Sik12]{Sik12}
Adam~S. Sikora.
\newblock Character varieties.
\newblock {\em Trans. Am. Math. Soc.}, 364(10):5173--5208, 2012.

\bibitem[Sim92]{Sim92}
Carlos~T. Simpson.
\newblock Higgs bundles and local systems.
\newblock {\em Publ. Math., Inst. Hautes {\'E}tud. Sci.}, 75:5--95, 1992.

\bibitem[SS03]{SS03}
Stefan Schwede and Brooke Shipley.
\newblock Stable model categories are categories of modules.
\newblock {\em Topology}, 42(1):103--153, 2003.

\bibitem[{Sta}25]{stacks-project}
The {Stacks project authors}.
\newblock The stacks project.
\newblock \url{https://stacks.math.columbia.edu}, 2025.

\bibitem[TV07]{TV07}
Bertrand To{\"e}n and Michel Vaqui{\'e}.
\newblock Moduli of objects in dg-categories.
\newblock {\em Ann. Sci. {\'E}c. Norm. Sup{\'e}r. (4)}, 40(3):387--444, 2007.

\bibitem[TV08]{TV08}
Bertrand To{\"e}n and Gabriele Vezzosi.
\newblock {\em Homotopical algebraic geometry. {II}: {Geometric} stacks and
  applications}, volume 902 of {\em Mem. Am. Math. Soc.}
\newblock Providence, RI: American Mathematical Society (AMS), 2008.

\bibitem[Wei60a]{Wei60}
Andr{\'e} Weil.
\newblock On discrete subgroups of {Lie} groups.
\newblock {\em Ann. Math. (2)}, 72:369--384, 1960.

\bibitem[Wei60b]{Wei62}
Andr{\'e} Weil.
\newblock On discrete subgroups of {Lie} groups. {II}.
\newblock {\em Ann. Math. (2)}, 75:578–602, 1960.

\bibitem[Wei94]{Wei94}
Charles~A. Weibel.
\newblock {\em An introduction to homological algebra}, volume~38 of {\em Camb.
  Stud. Adv. Math.}
\newblock Cambridge: Cambridge University Press, 1994.

\end{thebibliography}
